\documentclass[12pt]{amsart}
\usepackage{amsmath,amscd,amsbsy,amssymb,amsthm,amsfonts}
\usepackage{latexsym,url,bm}
\usepackage{cite,enumitem}
\usepackage{fullpage}
\usepackage{xcolor}
\usepackage[T1]{fontenc}
\usepackage[utf8]{inputenc}
\usepackage{lmodern}
\usepackage[colorlinks=true,allcolors=black]{hyperref}
\newif\ifmarkedversion
\markedversionfalse

\expandafter\def\csname revisedbibauthor@CHN\endcsname{}
\AtBeginDocument{\let\originalbibitem\bibitem
  \let\originalnewblock\newblock
  \renewcommand{\bibitem}[1]{\originalbibitem{#1}%
    \ifmarkedversion\ifcsname revisedbibauthor@#1\endcsname\color{blue}\fi\fi}
  \renewcommand{\newblock}{\ifmarkedversion\color{black}\fi\originalnewblock}}
\usepackage{chngcntr}
\usepackage{apptools}
\usepackage{graphicx}

\newtheorem{theorem}{Theorem}[section]
\newtheorem{proposition}[theorem]{Proposition}
\newtheorem{lemma}[theorem]{Lemma}

\newtheorem{remark}{Remark}[section]

\theoremstyle{definition}
\newtheorem{definition}[theorem]{Definition}

\allowdisplaybreaks

\numberwithin{equation}{section}

\newcommand{\p}{\partial}
\newcommand{\e}{\epsilon}
\newcommand{\wt}{\widetilde}
\newcommand{\R}{\mathbb{R}}

\newcommand{\Z}{\mathbb{Z}}

\newcommand{\al}{\alpha}
\newcommand{\f}{\frac}

\keywords{Hele--Shaw flow, Muskat-type problem, elastic interface,
Dirichlet--Neumann operator, well-posedness}
\subjclass[2020]{35R35, 35Q35, 35A01, 35A02}
\author{Lizhe Wan}
\address{Beijing International Center for Mathematical Research, Peking University}
\curraddr{}
\email{wanlizhe@pku.edu.cn}

\author{Jiaqi Yang}
\address{School of Mathematics and Statistics, Northwestern Polytechnical University}
\curraddr{}
\email{yjqmath@nwpu.edu.cn, yjqmath@163.com}

\begin{document}

\title{Well-Posedness of the Two-Dimensional Hele-Shaw flow with an Elastic Interface}

\begin{abstract}
We investigate a two-dimensional Hele-Shaw free-boundary problem with a nonlinear elastic interface, for both one-phase and two-phase scenarios.
Using the paradifferential formulation, we demonstrate that the problem is locally well-posed in $H^s$ for $s\geq 2$ for arbitrary initial data.
Furthermore, for the one-phase case and the stable two-phase case $(\rho^+ \leq \rho^-)$, we establish global well-posedness for small initial data in $H^s$ when $s> \f32$.
\end{abstract}

\maketitle
\tableofcontents

\section{Introduction} \label{s:Intro}

Darcy and Hele-Shaw type free-boundary problems provide fundamental models for the motion of incompressible fluids in porous media or thin gaps between closely spaced plates. In this paper, we study a two-dimensional Hele-Shaw flow in which the fluid motion is governed by Darcy's law and the interface separating the fluid regions is modeled as a thin elastic membrane carrying bending energy, a model referred to as the \textbf{Hele-Shaw flow with an elastic interface}.

Fluid models coupled with elastic structures arise in several related contexts, including hydroelastic waves, the Peskin problem, and Muskat-type problems for viscoelastic filtration; see Cameron and Strain \cite{MR4673875}, Gahn \cite{MR4959949}, Meirmanov \cite{MR2863467}, and Plotnikov and Toland \cite{MR2812947}. In the Darcy/Hele--Shaw setting, flows coupled with deformable boundaries, elastic membranes, or elastic interfaces have also been studied in the physical and applied literature. Elastic-walled Hele--Shaw cells and displacement flows under elastic membranes were investigated in \cite{JFM15-I,JFM15-II}, while models in which the fluid-fluid interface itself is endowed with elastic bending rigidity appear in \cite{HLB12,PRE13}. The fourth-order restoring force considered below is motivated by the bending energy of thin elastic sheets, as in nonlinear hydroelastic wave models.

A closely related class of models concerns hydroelastic waves, which describe the interaction between elastic structures and hydrodynamic forcing; see \cite{MR2812947}. 
In this setting, the fluid is governed by the incompressible Euler equations.
For developments in numerical simulation, experimental studies, and applications of hydroelastic waves, we refer to P\u{a}r\u{a}u \textit{et al.} \cite{MR2812939}. 
The local well-posedness of two-dimensional hydroelastic waves was established by Ambrose and Siegel \cite{MR3656704} and by Liu and Ambrose \cite{MR3608168}. 
Results on local well-posedness for hydroelastic waves with vorticity in arbitrary spatial dimensions can be found in the work of the second author and Wang \cite{MR4104949}. 
Recently, the authors obtained a low-regularity well-posedness result for two-dimensional hydroelastic waves in \cite{WanY}.

Let us denote the interface between the fluids (or the fluid and the air) at time $t$ by $\Sigma_t$.
In this paper, we assume throughout that $\Sigma_t$ can be represented by the graph of a time-dependent function $\eta(t,x)$, so that
\begin{equation*}
    \Sigma_t = \{ (x, \eta(t,x)): x\in \R \}.
\end{equation*}
Along the interface $\Sigma_t$, a thin elastic sheet separates the fluid domain into the upper and lower regions, denoted by $\Omega^+_t$ and $\Omega_t^-$, respectively.
They are given by
\begin{align*}
& \Omega_t^+ = \{ (x, y)\in \R^2: \eta(t,x)<y< \underline{b}^+(x) \}, \\
& \Omega_t^- = \{ (x, y)\in \R^2: \underline{b}^-(x)<y< \eta(t,x) \},
\end{align*}
where $\underline{b}^+(x)$ and $\underline{b}^-(x)$ are parameterizations of the upper and lower parts of the fluid boundary, respectively:
\begin{equation*}
    \Gamma^\pm = \{ (x, \underline{b}^\pm(x)): x\in \R \}.
\end{equation*}
\begin{figure}[htbp] 
    \centering 
\includegraphics[width=0.7\textwidth]{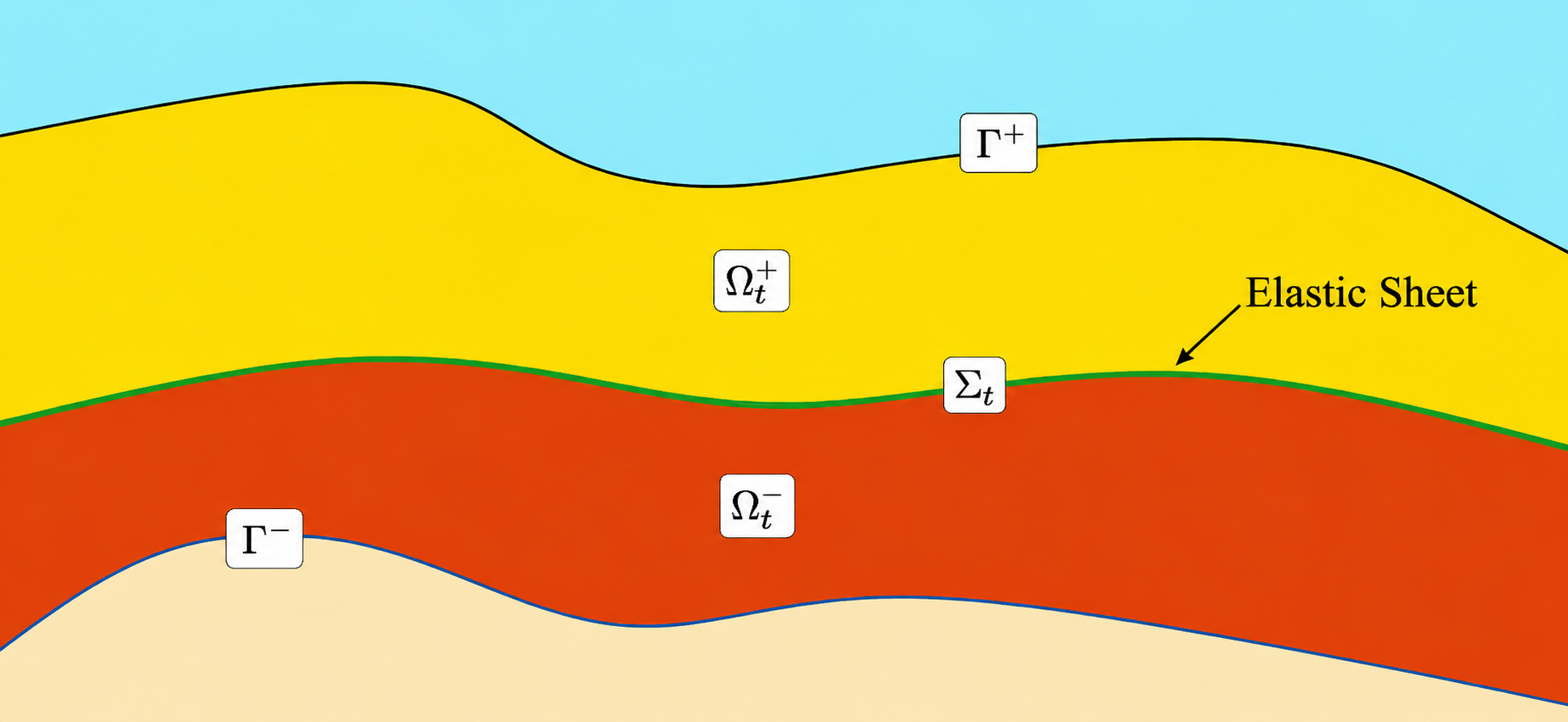}
\caption{Two-phase Hele-Shaw flow with an elastic interface} 
    \label{fig}
\end{figure}

The velocities $u^\pm$ and pressures $p^\pm$ of the incompressible fluids in the two regions are governed by Darcy's law:
\begin{equation} \label{Darcy}
\mu^\pm u^\pm + \nabla_{x,y}p^{\pm} = -\rho^{\pm}g \textbf{e}_y, \quad \nabla_{x,y}\cdot u^\pm = 0 \quad \text{in } \Omega^\pm_t, 
\end{equation}
where $g \geq 0$ denotes the gravitational acceleration, $\rho^\pm$ are the densities of the fluids in $\Omega^\pm$, and $\mu^\pm$ are the corresponding viscosity coefficients.

Along the interface $\Sigma_t$, we assume that the normal component of the velocity is continuous, namely,
\begin{equation} \label{NormalVelocity}
    u^+ \cdot n = u^- \cdot n \quad \text{on } \Sigma_t, 
\end{equation}
where $n = \frac{1}{\sqrt{1+\eta_x^2}}(-\eta_x, 1)$ denotes the upward-pointing unit normal vector to $\Sigma_t$.
Then the kinematic boundary condition on $\Sigma_t$ takes the form
\begin{equation} \label{kinematic}
    \p_t \eta = \sqrt{1+\eta_x^2} u^- \cdot n\big|_{\Sigma_t}.
\end{equation}
The dynamic boundary condition on $\Sigma_t$ asserts that the pressure jump across $\Sigma_t$ is balanced by the restoring force generated by the elastic sheet.
\begin{equation} \label{dynamic}
   p^- - p^+ = \sigma \mathbf{E}(\eta) \quad \text{on } \Sigma_t, 
\end{equation}
where $\sigma>0$ is the coefficient of flexural rigidity.
The nonlinear elastic operator $\mathbf{E}(\eta)$, arising from the bending energy of the sheet, is given by
\begin{equation} \label{Elastic}
\begin{aligned}
    \mathbf{E}(\eta) :=&\frac{1}{\sqrt{1+\eta_x^2}} 
    \left[ \frac{1}{\sqrt{1+\eta_x^2}} 
    \left( \frac{\eta_{xx}}{(1+\eta_x^2)^{3/2}} \right)_x \right]_x
    + \frac{1}{2} \left( \frac{\eta_{xx}}{(1+\eta_x^2)^{3/2}} \right)^3\\
    =&\left( \dfrac{1}{1+\eta_x^2}\left(\dfrac{\eta_x}{\sqrt{1+\eta_x^2}}\right)_x\right)_{xx}+\f52\left(\f{\eta_x\eta_{xx}^2}{(1+\eta^2_x)^{\f72}}\right)_x.
\end{aligned}
\end{equation}

On the rigid upper and lower boundaries, we assume the non-penetration condition,
\begin{equation} \label{NoPenetration}
    u^\pm \cdot (-\underline{b}^\pm_x, 1) = 0, \quad \text{on } \Gamma^\pm.
\end{equation}

The system \eqref{Darcy}-\eqref{NoPenetration} is referred to as the  \textbf{two-phase} Hele-Shaw flow with an elastic interface.
When the upper phase corresponds to air or vacuum, so that $\mu^+ = \rho^+ = p^+=0$, the problem reduces to the \textbf{one-phase} Hele-Shaw flow.

\subsection{Formulation of the Hele-Shaw flow with an elastic interface}
The Hele-Shaw/Muskat problem admits several equivalent formulations; see \cite{MR4520423,MR3415681,MR3171344,MR2313156,MR2993754,MR3608884,MR3071395}. 
In the absence of surface tension, C\'ordoba \& Gancedo \cite{MR2318314} introduced a contour dynamics approach for the infinite-depth Muskat problem without viscosity jump, and proved local well-posedness in $H^3$ for graph interfaces. 
This approach was later extended by C\'ordoba, C\'ordoba \& Gancedo \cite{MR2753607,MR3071395} to include a viscosity jump and to non-graph interfaces satisfying the arc-chord and Rayleigh--Taylor conditions.
Using an arbitrary Lagrangian--Eulerian method, Cheng, Granero \& Shkoller \cite{MR3415681} established local well-posedness for the one-phase problem with a flat bottom, assuming initial surfaces $\eta \in H^2$. Constantin, Gancedo, Shvydkoy \& Vicol \cite{CGSV17} obtained a local existence theory in the subcritical regime and proved global regularity under a finite-slope condition.
Matioc \cite{MR3861893} later refined the regularity requirement to $\eta \in H^s$ with $s > \frac32$ for the constant-viscosity, infinite-depth case. 
Alazard \& Lazar \cite{MR4097324} provided an alternative proof via paralinearization of the contour dynamics formulation.
 Nguyen \& Pausader \cite{MR4090462} reformulated the problem using the Dirichlet--Neumann operator and proved local well-posedness for large data in all subcritical Sobolev  spaces. Alazard \& Nguyen \cite{MR4242131} proved the well-posedness for any initial data in the critical space of Lipschitz functions with $H^{\f32}$ regularity. Later, in \cite{MR4313450}, they improved the result to the non-Lipschitz initial data. For similar results in the three-dimensional case and the two-phase case, see \cite{MR4387237} and \cite{MR4402655}, respectively. Alazard \& Nguyen \cite{MR4541917} further proved local well-posedness for
large data in the critical Sobolev space $H^{\f32}$. Recently, Cameron, Chen, Hu, Nguyen \&  Xu \cite{MR4986619} studied the Muskat problem with a large slope. 
For results on Hele-Shaw and Muskat problems with surface tension,
we refer to \cite{MR4131404,MR3841857,AgrawalPatel2025} and the references therein. 
We also mention the recent work of Chen, Hu \& Nguyen
\cite{ChenHuNguyen2026}, where Schauder-type estimates and log-critical
well-posedness were obtained for the two-phase Muskat problem with surface
tension.
For broader studies on non-local parabolic equations with application to the Muskat problem with surface tension, see \cite{CHN}.

Concerning global well-posedness, relevant works include \cite{MR3869383,MR3415681,MR4363243,MR3899970,MR3595492,MR2998834,Lazar,MR4348695,MR4655356}.

In this paper, we follow the formulation in Nguyen \cite{MR4131404} and Nguyen \& Pausader \cite{MR4090462}, and rewrite the present problem in terms of the Dirichlet-Neumann operators $G^\pm(\eta)$ associated with $\Omega^\pm_t$. 
For a fixed time $t$, and a given function $f$ defined on $\Sigma$, we let $\phi^\pm$ be the solution of 
\begin{equation*}
\begin{cases}
 \Delta_{x,y} \phi^\pm = 0 \quad \text{in }\Omega^\pm, \\
\phi^\pm = f \quad \text{on }\Sigma, \\
\frac{\p \phi^\pm}{\p\nu^\pm} = 0 \quad \text{on }\Gamma^\pm.
\end{cases}
\end{equation*}
The Dirichlet--Neumann operators $G^\pm(\eta)$ are given by
\begin{equation*}
 G^\pm(\eta) f : = \sqrt{1+\eta_x^2} \frac{\p \phi^\pm}{\p n}.  
\end{equation*}

Then the two-dimensional Hele-Shaw flow with an elastic interface can be reformulated using the following result.

\begin{proposition}[\hspace{1sp}\cite{MR4131404}] \label{t:Formuation}
$(i)$ If $(u,p, \eta)$ is a solution to the one-phase problem, then $\eta$ solves the differential equation
\begin{equation} \label{OneMuskat}
    \p_t \eta = -\frac{1}{\mu^-}G^-(\eta)(\sigma \mathbf{E}(\eta)+ \rho^- g \eta).
\end{equation}
On the other hand, if $\eta$ is a solution of the differential equation \eqref{OneMuskat}, then the one-phase problem has a solution, in which $\eta$ parameterizes the free surface $\Sigma_t$.

$(ii)$ If $(u^\pm, p^\pm, \eta)$ is a solution of the two-phase problem, then $\eta$ satisfies
\begin{equation} \label{TwoMuskatOne}
    \p_t \eta = -\frac{1}{\mu^-}G^-(\eta)f^-,
\end{equation}
where $f^\pm := p^\pm|_{\Sigma_t} + \rho^\pm g\eta$ satisfy 
\begin{equation}  \label{TwoMuskatTwo}
\begin{cases}
f^- - f^+ = \sigma \mathbf{E}(\eta) + g(\rho^- - \rho^+)\eta, \\
\frac{1}{\mu^+}G^+(\eta)f^+ = \frac{1}{\mu^-}G^-(\eta)f^-.
\end{cases}
\end{equation}
On the other hand, if $\eta$ is a solution of \eqref{TwoMuskatOne} where $f^\pm$ solve \eqref{TwoMuskatTwo}, then the two-phase  problem has a solution, in which $\eta$ parameterizes the free surface $\Sigma_t$.
\end{proposition}

Proposition \ref{t:Formuation} was proved in Appendix $B$ in \cite{MR4131404} for the Muskat problem with surface tension.
Replacing the capillary terms by the elastic terms yields the corresponding result stated above.

The Hele-Shaw flow with an elastic interface is essentially a fifth-order quasilinear parabolic equation.
Indeed, at the leading order, one may intuitively approximate $\sigma \mathbf{E}(\eta) + \rho^- g\eta\approx \sigma  \Delta^2 \eta$, so that the one-phase problem \eqref{OneMuskat} takes the schematic form
\begin{equation} \label{ToyModel}
     \p_t \eta = -\frac{\sigma}{\mu^-}G^-(\eta)\Delta^2 \eta + \text{remainder terms}.
\end{equation}
In the bottomless case where $\Gamma^\pm = \emptyset$, if $\eta(t,x)$ solves \eqref{ToyModel}, then so does 
\begin{equation*}
    \eta_\lambda (t,x) = \lambda^{-1}\eta(\lambda^5 t, \lambda x), \quad \forall \lambda >0.
\end{equation*}
Hence, the scale-invariant Sobolev space is $\dot{H}^\f32(\R)$.
This critical regularity coincides with that of the gravity Muskat problem and the Muskat problem with surface tension.
As in the case of water waves with surface tension \cite{MR2805065} and the Muskat problem with surface tension \cite{MR4131404}, the presence of rigid boundaries $\Gamma^\pm$ affects only the low-frequency behavior and does not alter the local well-posedness theory.

\subsection{Main results}
Before stating the main results of this paper, we first introduce the functional spaces used throughout the paper.
We define
\begin{align*}
    &\dot{W}^{1,\infty} : = \{u\in L^1_{loc}(\R) : \nabla u \in L^\infty(\R)  \}, \\
    & Z^s(T) : = C([0,T]; H^s(\R)) \cap L^2([0,T]; H^{s+\f52}(\R)).
\end{align*}

Our first main result is on the local well-posedness of two-dimensional Hele-Shaw flow with an elastic interface for large data in $H^s$ for $s \geq 2$.
\begin{theorem} \label{t:MainOne}
Let $s\geq 2$, and let $\sigma>0$, $g\geq 0$ be fixed parameters.

$(i)$ \textbf{Local well-posedness for the one-phase problem}: Let $\mu^- >0$, and $\rho^->0$.
We consider either $\Gamma^- = \emptyset$ or $\underline{b}^- \in \dot{W}^{1,\infty}$.
Let $\eta_0 \in H^s(\R) ($\text{or} $H^s(\mathbb{T}))$   satisfy 
\begin{equation*}
    \text{dist}(\eta_0, \Gamma^-)>2h>0.
\end{equation*}
Then there exists a time $T>0$, depending only on $\|\eta_0 \|_{H^s}$, $h$, $s$, $\frac{\rho^- g}{\mu^-}$ and $\frac{\sigma}{\mu^-}$, and a unique solution $\eta \in Z^s(T)$ to the equation \eqref{OneMuskat} such that
\begin{equation*}
    \eta(0,x) = \eta_0, \quad \inf_{t\in[0,T]}\text{dist}(\eta(t), \Gamma^-)>h.
\end{equation*}
Moreover, if $\eta_1$ and $\eta_2$ are two solutions of \eqref{OneMuskat}, then 
\begin{equation} \label{StabilityOne}
    \|\eta_1 - \eta_2 \|_{Z^s(T)} \lesssim_{\|\eta_1\|_{Z^s(T)} + \|\eta_2\|_{Z^s(T)}} \|(\eta_1 - \eta_2)(0,\cdot) \|_{H^s}.
\end{equation}

$(ii)$ \textbf{Local well-posedness for the two-phase problem}: Let $\mu^\pm >0$, and $\rho^\pm>0$.
We consider either $\Gamma^\pm = \emptyset$ or $\underline{b}^\pm \in \dot{W}^{1,\infty}$.
Let $\eta_0 \in H^s(\R) ($\text{or} $H^s(\mathbb{T}))$ satisfy 
\begin{equation*}
    \text{dist}(\eta_0, \Gamma^\pm)>2h>0.
\end{equation*}
Then there exists a time $T>0$, depending only on $\|\eta_0 \|_{H^s}$, $h$, $s$, $\mu^\pm$, $\sigma$ and $g(\rho^- -\rho^+)$, and a unique solution $\eta \in Z^s(T)$ to  \eqref{TwoMuskatOne}-\eqref{TwoMuskatTwo} such that
\begin{equation*}
    \eta(0,x) = \eta_0, \quad \inf_{t\in[0,T]}\text{dist}(\eta(t), \Gamma^\pm)>h.
\end{equation*}
Moreover, if $\eta_1$ and $\eta_2$ are two solutions of \eqref{TwoMuskatOne}-\eqref{TwoMuskatTwo}, then 
\begin{equation} \label{StabilityTwo}
    \|\eta_1 - \eta_2 \|_{Z^s(T)} \lesssim_{\|\eta_1\|_{Z^s(T)} + \|\eta_2\|_{Z^s(T)}} \|(\eta_1 - \eta_2)(0,\cdot) \|_{H^s}.
\end{equation}
\end{theorem}

We make the following remarks about the local well-posedness result.
\begin{remark}
i) Theorem \ref{t:MainOne} establishes well-posedness in the sense of Hadamard, namely: existence, uniqueness and Lipschitz dependence on initial data.
To the best of our knowledge, these are the first results on the Cauchy theory of two-dimensional Hele-Shaw flow with an elastic interface.
In particular, these are large-data well-posedness results that require only $\f12$ derivative above the critical regularity.
Under this regularity setting, the curvature of the free surface $\Sigma_t$ may be unbounded.

ii) 
Let us explain why the Sobolev regularity threshold $s\ge 2$ is needed for the large-data local well-posedness theory presented above. The point is the $H^{\f12}$-input threshold in the low-regularity remainder estimate for the Dirichlet--Neumann operator recalled in Appendix~B, see \eqref{GfDRf}.
More precisely, in the paralinearization 
\[ G^{-}(\eta)f=|D|f+R^-(\eta)f, \] 
the estimate for the remainder term $R^-(\eta)f$ is available in our framework only when the input $f$ has Sobolev regularity at least $1/2$.

In the classical Muskat problem, as studied for instance by Nguyen--Pausader
\cite{MR4090462}, the Dirichlet--Neumann operator acts on a lower-order
boundary quantity. More precisely, up to physical constants, the equation has the schematic form
\[
  \partial_t \eta = -G^{-}(\eta)\eta.
\]

Thus, if $\eta\in H^s$ with $s>3/2$, then the input of the
Dirichlet--Neumann operator is already well above the $H^{\f12}(\R)$ threshold
required in the low-regularity remainder estimates for $G^{-}(\eta)$.

For the Muskat problem with surface tension, treated by Nguyen \cite{MR4131404},
the Dirichlet--Neumann operator acts on the curvature term. 
Since $\kappa(\eta)\sim \partial_x^2\eta$, and then $G^{-}(\eta)\kappa(\eta)\sim |D|^3\eta$, the equation has a third-order parabolic structure. In the corresponding
parabolic energy space, one gains $3/2$ derivatives in spacetime, so the
curvature is naturally controlled at the level
\[
  \kappa(\eta)\in L^2([0,T];H^{s-\frac12}(\R)).
\]
For every $s>3/2$, this regularity is above $H^{1/2}$. Hence the
Dirichlet--Neumann remainder estimates can still be applied throughout the
subcritical range.

The situation is different in the present elastic-interface problem. Here the
pressure jump contains the fourth-order nonlinear bending force $ \mathbf{E}(\eta)$.
Schematically,
\[
  \mathbf{E}(\eta)\sim \partial_x^4\eta,
  \qquad
  G^{-}(\eta) \mathbf{E}(\eta)\sim |D|^5\eta.
\]
Thus the leading part is formally fifth-order parabolic. In the energy space
\[
  Z^s(T)=C([0,T];H^s(\R))\cap L^2([0,T];H^{s+\frac52}(\R)),
\]
the elastic force is naturally controlled only at the level
\[
  \mathbf{E}(\eta)\in L^2([0,T];H^{s-\frac32}(\R)).
\]
When applying the paralinearization of the Dirichlet--Neumann operator to \(f=\mathbf E(\eta)\), the remainder estimate \eqref{GfDRf} requires 
\[ s-\frac32\ge \frac12, \] 
which gives \(s\ge2\). 

This is the main regularity difference between the classical Muskat problem, the capillary Muskat problem, and the present elastic-interface problem.
The restriction $s\ge2$ should therefore be viewed as a limitation of the
present large-data quasilinear paradifferential method, rather than as a
scaling obstruction. Extending the arbitrary-data local theory to the full
range $s>3/2$ would require sharper low-regularity estimates for the
Dirichlet--Neumann operator, or a different formulation adapted to the
fourth-order elastic force. However, for sufficiently small initial data, the present problem can be treated as a semilinear parabolic equation.
By exploiting the parabolic spacetime estimates \eqref{SpaceTimeOne} and \eqref{SpaceTimeTwo}, one can establish the global well-posedness of Hele-Shaw flow with an elastic interface in all sub-critical spaces, as stated in Theorem \ref{t:global} below.

iii) As a quasilinear parabolic equation, the Hele-Shaw flow with an elastic interface exhibits instantaneous parabolic smoothing.
If $\eta \in C([0,T]; H^s(\R))$, then $\eta(t, \cdot) \in C^\infty(\R)$ when $t\in [\delta, T]$, for any $\delta>0$.
The proof is standard and therefore omitted. 
See \cite[Section 9]{MR3415681} for an analogous argument in the gravity Muskat setting.

\end{remark}

We define the space $\widetilde{L}^q(I; H^s(\R))$ as the space of tempered distributions $u \in \mathcal{S}' (\R^2)/ \mathcal{P}(\R)$ such that
\begin{equation*}
 \|u\|_{\widetilde{L}^q(I; H^s(\R))} : = \|2^{sj} \|P_j u \|_{L^q(I; L^2(\R))} \|_{\ell^2(\Z)} < \infty.
\end{equation*}
Our second main result is on the global well-posedness of  two-dimensional Hele-Shaw flow with an elastic interface in all sub-critical Sobolev spaces for small initial data.
 \begin{theorem}\label{t:global}
Let $s>\f32$, and suppose that the boundaries $\Gamma^\pm$ are either empty or flat $(\underline{b}^{\pm}_x = 0)$. 
We consider either the one-phase  or the  two-phase Hele-Shaw flow with an elastic interface in the stable regime $\rho^+\leq \rho^-$. 
Let $\eta_0\in H^s(\R)$ be an initial datum. Then there exist positive constants $\delta$ and $C$, such that the following holds: if $\| \eta_0\|_{H^s}\le \delta$ then for any $T>0$, \eqref{OneMuskat} and \eqref{TwoMuskatOne}-\eqref{TwoMuskatTwo} admit a unique solution $\eta\in C([0, T]; H^s (\R))$ satisfying
\begin{equation*}
\| \eta\|_{\wt L^\infty([0, T]; H^s)}+\frac{\sigma}{\mu^-} \| \eta\|_{\wt L^1([0, T];    H^{s+5})}\le C\| \eta_0\|_{H^s}.
\end{equation*}
 Moreover, if $\eta^1$ and $\eta^2$ are two solutions to the corresponding
problem, then for all $T>0$,
\begin{equation*}
\| \eta^1-\eta^2\|_{\wt L^\infty([0, T]; H^s)}+\frac{\sigma}{\mu^-} \| \eta^1-\eta^2\|_{\wt L^1([0, T]; H^{s+5})}\le C\| \eta^1(0)-\eta^2(0)\|_{H^s}.
\end{equation*}
\end{theorem}

For the two-phase problem, the assumption $\rho^+ \leq  \rho^-$ corresponds to the situation where the  denser fluid lies below the less dense fluid, so that it is physically stable.
While local well-posedness holds regardless of density ratios, global well-posedness for small data is established here specifically for the stable regime $\rho^+ \leq  \rho^-$.
It is not clear whether or not one can obtain the global well-posedness result in the unstable regime for small data.
In the classical gravity Muskat problem without an elastic interface, the regime $\rho^+>\rho^-$ is known to be ill-posed in Sobolev spaces due to the heavier fluid being situated above the lighter one, leading to the rapid growth of interface perturbations.
See C\'ordoba \& Gancedo \cite{MR2318314} and \cite{MR2993754}, C\'ordoba, G\'omez-Serrano \& Zlato\v{s} \cite{MR3393318,MR3619874} for detailed results. 
In addition, for the Muskat problem with surface tension,  Lazar \cite{Lazar} established the global well-posedness in the strictly stable case $\rho^+<\rho^-$. We also mention that, in certain periodic or finite-depth settings, surface tension 
may stabilize Rayleigh--Taylor unstable configurations when it is sufficiently strong; 
see \cite{GGSS20}.

We summarize the main results established in the paper in Table~\ref{tab:muskat_results}.
\begin{table}[htbp]
\centering
\renewcommand{\arraystretch}{1.5}
\begin{tabular}{|l|l|c|l|}
\hline
\textbf{Phase} & \textbf{Result type} & \textbf{Regularity ($H^s(\R)$)} & \textbf{Data constraints} \\ \hline
one/two-phase & local well-posedness & $s \ge 2$ & arbitrary large data \\ \hline
one-phase & global well-posedness & $s > \frac{3}{2}$ & small initial data \\ \hline
two-phase (stable) & global well-posedness & $s > \frac{3}{2}$ & small data; $\rho^+ \le \rho^-$ \\ \hline
\end{tabular}
\caption{Well-posedness results  established in this paper.}
\label{tab:muskat_results}
\end{table}

After completing the first version of this manuscript, the authors became aware
of the closely related work of Alonso-Or\'an and Granero-Belinch\'on \cite{porous}.
They derived asymptotic interface models for an elastic Muskat free-boundary
problem describing Darcy flow beneath an elastic membrane, and established
well-posedness for these models in suitable Wiener spaces.

Let us clarify the relation between the present work and the existing
Dirichlet--Neumann/paradifferential theory for the Muskat problem. Our general
formulation and strategy are inspired by the work of Nguyen \cite{MR4131404} on
the Muskat problem with surface tension and the work of Nguyen--Pausader \cite{MR4090462} on the classical Muskat problem. In particular, we use the
Dirichlet--Neumann formulation, paradifferential calculus, energy estimates,
contraction estimates, and approximation arguments in the spirit of these works.

The present elastic-interface problem, however, has a genuinely different
high-order structure. In the surface-tension problem, the pressure jump involves
the curvature, which is a second-order quantity. In contrast, here the pressure
jump contains the nonlinear fourth-order elastic bending force $\mathbf E(\eta)$.
After composition with the first-order Dirichlet--Neumann operator, this produces
a fifth-order quasilinear parabolic equation. Thus the elastic force is not
merely a higher-order analogue of the surface-tension term: it creates new
high-order nonlinear remainders and changes the regularity balance in both the
energy estimates and the contraction estimates.

The first new ingredient is a detailed paralinearization of the nonlinear elastic
operator $\mathbf E(\eta)$. This paralinearization identifies a fourth-order
elliptic paradifferential operator $T_\ell$, with principal symbol
\[
  \ell^{(4)}(x,\xi)=(1+\eta_x^2)^{-\frac52}\xi^4,
\]
and separates all remaining terms into remainders compatible with the
large-data energy space. 
This step is specific to the bending energy of the elastic interface and has no direct counterpart in the classical Muskat problem.

The second new ingredient is the derivation of the fifth-order quasilinear paradifferential parabolic structure.
Combining the above paralinearization with
the leading part of the Dirichlet--Neumann operator gives the fifth-order elliptic operator
\[
  |D|T_\ell,
  \qquad
  \text{with principal paradifferential symbol }
  (1+\eta_x^2)^{-\frac52}|\xi|^5.
\]
The energy estimates are built to exploit this ellipticity while controlling the lower-order terms generated by the variable coefficient symbol and by the Dirichlet--Neumann remainders.

A further point concerns the two-phase problem. 
In this case, the pressure traces $f^\pm$ are determined by a coupled elliptic system involving the elastic jump condition. 
We prove estimates for these traces with the fourth-order elastic forcing and then derive the corresponding fifth-order paradifferential equation for the interface.

Finally, to prove well-posedness for arbitrary data, we establish contraction estimates adapted to the fourth-order elastic force. 
These estimates require controlling differences of $\mathbf E(\eta)$, of the Dirichlet--Neumann
operators, and of the two-phase pressure traces at the same high-order regularity level as the energy estimates. 
They are the key estimates leading to uniqueness and Lipschitz dependence on the initial data. 
For the small-data global theory, we further rewrite the equation in integral form around the fifth-order fractional heat semigroup and use parabolic space-time estimates to close a fixed-point argument in all subcritical Sobolev spaces.

The rest of this paper is organized as follows.
In Section \ref{s:Reduct}, we first paralinearize the elastic term. 
Then we rewrite the one-phase and two-phase problems as paradifferential parabolic equations and obtain a priori energy estimates for these equations.
In Section \ref{s:Contraction}, we first prove contraction estimates for both one-phase and two-phase problems, then we combine energy estimates and contraction estimates to obtain local well-posedness of Hele-Shaw flow with an elastic interface.
Section \ref{s:Global} is then devoted to the proof of global well-posedness for   Hele-Shaw flow with an elastic interface.
We rewrite the solution of Hele-Shaw flow with an elastic interface as a fixed point of an integral equation, and use the fixed-point Lemma \ref{t:fixedpoint} to obtain existence and uniqueness of the global solution.
Many necessary paradifferential and product estimates are recalled in Appendix \ref{s:Paradifferential}.
Finally, we give some results on the Dirichlet-Neumann operators, the fractional heat kernel and the existence of fixed points in Appendix \ref{s:DNOperator}.

\section{Reduction and a priori estimates} \label{s:Reduct}
In this section we reformulate the Hele-Shaw flow with an elastic interface using paradifferential calculus and derive a priori estimates for the Hele-Shaw flow with an elastic interface. 
Throughout this section, we assume that $\eta \in Z^s(T)$ for $s\geq2$, and the surface remains at a positive distance from the boundaries $\Gamma^\pm$.
\begin{equation} \label{BottomH}
 \inf_{t\in[0,T]}\text{dist}(\eta(t), \Gamma^\pm)> h>0.
\end{equation}

We begin by paralinearizing the elastic term defined in $\mathbf{E}(\eta)$ in \eqref{Elastic}.
\begin{lemma} 
Let $s\geq2$, and $0<\delta \leq \f12$.
The elastic term admits the decomposition 
\begin{equation} \label{ElasticPara}
  \mathbf{E}(\eta) = T_{\ell} \eta + R_\mathbf{E}(\eta),   
\end{equation}
where $ T_{\ell} \eta$ is the principal paradifferential term, and the symbol $\ell$ is given by 
\begin{equation} \label{EllDef}
\begin{aligned}
    &\ell(x, \xi) = (1+\eta_x^2)^{-\f52}\xi^4 -2i((1+\eta_x^2)^{-\f52})_x \xi^3 \\
    -& \left(((1+\eta_x^2)^{-\f52})_{xx} -5(\eta_{xx}\eta_x(1+\eta_x^2)^{-\f72})_{x} +\f52\eta_{xx}^2(1-6\eta_x^2)(1+\eta_x^2)^{-\f92}\right) \xi^2 \\
    +& i\left(\f52 (\eta_{xx}^2(1-6\eta_x^2)(1+\eta_x^2)^{-\f92})_x-5(\eta_{xx}\eta_x(1+\eta_x^2)^{-\f72})_{xx}\right)\xi.
\end{aligned}    
\end{equation}
The lower-order remainder term $R_\mathbf{E}(\eta)$ satisfies the estimate
\begin{equation} \label{REEtaEst}
    \|R_\mathbf{E}(\eta)\|_{H^{s-\f32+\delta}} \lesssim \|\eta\|_{H^s} \|\eta\|_{H^{s+\f52}}. 
\end{equation}
\end{lemma}

\begin{proof}
Using the paralinearization \eqref{Paralinearization}, we can write
\begin{equation*}
\dfrac{\eta_x}{\sqrt{1+\eta_x^2}} =  T_{(1+\eta_x^2)^{-\f32}} \eta_x + R_1, \quad \dfrac{\eta_x}{(1+\eta_x^2)^\f72} =  T_{(1-6\eta_x^2)(1+\eta_x^2)^{-\f92}} \eta_x + R_2
\end{equation*}
where the remainder terms $R_1$ and $R_2$ satisfy
\begin{equation*}
 \|R_1\|_{H^{s+\f12 + \delta}} + \|R_2\|_{H^{s+\f12 + \delta}}\lesssim_{\|\eta_x\|_{L^\infty}} \| \eta_x\|_{H^{s+\f32}} \|\eta_x\|_{C^\delta_*} \lesssim \|\eta\|_{H^s} \| \eta\|_{H^{s+\f52}}.
\end{equation*}
Using \eqref{BalancedError}, we obtain
\begin{align*}
 &\left(\dfrac{\eta_x}{\sqrt{1+\eta_x^2}}\right)_x = T_{(1+\eta_x^2)^{-\f32}} \eta_{xx} -3T_{\eta_{xx}\eta_x(1+\eta_x^2)^{-\f52}} \eta_{x} + R_3, \\
 &\|R_3\|_{H^{s+\f12 + \delta}}\lesssim \|\eta\|_{H^s} \| \eta\|_{H^{s+\f52}}, \\
 &\dfrac{1}{1+\eta_x^2}\left(\dfrac{\eta_x}{\sqrt{1+\eta_x^2}}\right)_x = T_{(1+\eta_x^2)^{-\f52}} \eta_{xx} -5T_{\eta_{xx}\eta_x(1+\eta_x^2)^{-\f72}} \eta_{x} + R_4, \\
 &\|R_4\|_{H^{s+\f12 + \delta}}\lesssim \|\eta\|_{H^s} \| \eta\|_{H^{s+\f52}}, \\
 &\f{\eta_x\eta_{xx}^2}{(1+\eta^2_x)^{\f72}} = 2T_{\eta_{xx}\eta_x(1+\eta_x^2)^{-\f72}}\eta_{xx} + T_{\eta_{xx}^2(1-6\eta_x^2)(1+\eta_x^2)^{-\f92}} \eta_x + R_5, \\
 &\|R_5\|_{H^{s-\f12 + \delta}}\lesssim \|\eta\|_{H^s} \| \eta\|_{H^{s+\f52}}.
\end{align*}
Taking spatial derivatives of these terms yields the following identities, with a remainder $R$ satisfying \eqref{REEtaEst},
\begin{align*}
    &\left( \dfrac{1}{1+\eta_x^2}\left(\dfrac{\eta_x}{\sqrt{1+\eta_x^2}}\right)_x\right)_{xx} = T_{(1+\eta_x^2)^{-\f52}} \p_x^4\eta + 2T_{((1+\eta_x^2)^{-\f52})_x} \p_x^3\eta + T_{((1+\eta_x^2)^{-\f52})_{xx}} \eta_{xx}\\
    -&5T_{\eta_{xx}\eta_x(1+\eta_x^2)^{-\f72}} \p_x^3\eta -10T_{(\eta_{xx}\eta_x(1+\eta_x^2)^{-\f72})_x} \eta_{xx} -5T_{(\eta_{xx}\eta_x(1+\eta_x^2)^{-\f72})_{xx}} \eta_x + R, \\
    &\f52\left(\f{\eta_x\eta_{xx}^2}{(1+\eta^2_x)^{\f72}}\right)_x = 5T_{\eta_{xx}\eta_x(1+\eta_x^2)^{-\f72}}\p_x^3\eta + 5T_{(\eta_{xx}\eta_x(1+\eta_x^2)^{-\f72})_x}\eta_{xx}\\
    +& \f52 T_{\eta_{xx}^2(1-6\eta_x^2)(1+\eta_x^2)^{-\f92}} \eta_{xx}  + \f52 T_{(\eta_{xx}^2(1-6\eta_x^2)(1+\eta_x^2)^{-\f92})_x} \eta_{x} + R.
\end{align*}
Adding these two terms together, we get the paralinearization of the elastic term \eqref{ElasticPara}.
\end{proof}

We emphasize that the precise explicit expression of the symbol $\ell$ is not essential to the subsequent analysis.
What is crucial is its elliptic structure.
In particular, for $u\in H^{s+4}$, $s> \f32$, the symbol $\ell$ satisfies
\begin{equation} \label{ellEst}
    \|T_{\ell} u - \p_x^4 u\|_{H^s} \lesssim \|\eta_x\|_{C^\epsilon_*} \|u \|_{H^{s+4}} \lesssim \|\eta\|_{H^s} \| \eta\|_{H^{s+4}}.
\end{equation}

\subsection{The one-phase case}
The one-phase problem \eqref{OneMuskat} is reformulated into a paradifferential form to facilitate energy estimates.
\begin{proposition}
For $\delta \in (0, \f12]$, the one-phase problem \eqref{OneMuskat} can be rewritten as 
\begin{equation} \label{ParaOneMuskat}
 \p_t \eta = -\frac{\sigma}{\mu^{-}} |D|T_\ell \eta + K, 
\end{equation}
where the remainder term $K$ is controlled by the initial data and gravitational terms as follows:
\begin{equation*}
    \|K\|_{H^{s-\f52+\delta}} \lesssim_{\|\eta\|_{H^s}} \frac{\sigma}{\mu^{-}}  \|\eta\|_{H^{s+\f52}} + \frac{\rho^- g}{\mu^-}\| \eta\|_{H^{s-\f32+\delta}}.
\end{equation*}
\end{proposition}

\begin{proof}
Recall the one-phase Hele-Shaw flow with an elastic interface reads
\begin{equation*}
  \p_t \eta = -\frac{\sigma}{\mu^-}G^-(\eta) \mathbf{E}(\eta) -\frac{\rho^-g}{\mu^-} G^-(\eta)\eta.
\end{equation*}
For the term $G^-(\eta)\eta$, the estimate \eqref{GfEtaf} with $\mathfrak{s} = s-\f32+\delta$ yields
\begin{equation*}
   \| G^-(\eta)\eta\|_{H^{s-\f52+\delta}} \lesssim_{\|\eta\|_{H^s}} \|\eta\|_{H^{s-\f32+\delta}}.
\end{equation*}
Regarding the term $G^-(\eta) \mathbf{E}(\eta)$, we use \eqref{GfDRf}, and \eqref{GfEtaf} to write
\begin{equation*}
 G^-(\eta) \mathbf{E}(\eta) = |D|\mathbf{E}(\eta) + R^{-}(\eta)\mathbf{E}(\eta),
\end{equation*}
where using \eqref{GfDRf},
\begin{equation*}
 \| R^{-}(\eta)\mathbf{E}(\eta) \|_{H^{s-\f52+\delta}} \lesssim_{\|\eta\|_{H^s}} \|\mathbf{E}(\eta) \|_{H^{s-\f32}} \lesssim_{\|\eta\|_{H^s}} \|\eta\|_{H^{s+\f52}},
\end{equation*}
so that this term can be put into $K$.
For the $|D|\mathbf{E}(\eta)$ term, we use \eqref{ElasticPara} to estimate the difference
\begin{equation*}
\||D|\mathbf{E}(\eta) - |D|T_{\ell}\eta \|_{H^{s-\f52+\delta}} \lesssim  \|\mathbf{E}(\eta) - T_{\ell}\eta \|_{H^{s-\f32+\delta}}\lesssim_{\|\eta\|_{H^s}} \|\eta\|_{H^{s+\f52}}.    
\end{equation*}
By putting the perturbative terms into $K$, we obtain the estimate \eqref{ParaOneMuskat}.
\end{proof}

At leading order, the symbol $|\xi|\ell(x,\xi) \approx (1+\eta_x^2)^{-\f52}|\xi|^5$, which is elliptic of order $5$.
Consequently, the one-phase problem \eqref{ParaOneMuskat} is a fifth-order quasilinear parabolic equation.

We have the following a priori estimate for the one-phase equation.
\begin{proposition} \label{t:OnePhaseEst}
Let $s\geq2$, and assume that $\eta \in Z^s(T)$ is a solution to \eqref{ParaOneMuskat} that satisfies \eqref{BottomH}.
Then there exists a function $\mathcal{F}: \R^+ \rightarrow \R^+$ depending only on $(h,s,\frac{\rho^- g}{\mu^-},\frac{\sigma}{\mu^-})$ such that
\begin{equation} \label{OneAprioriEst}
  \| \eta\|_{Z^s(T)} \lesssim \mathcal{F}\big(\|\eta(0,\cdot)\|_{H^s} + T^\f12\mathcal{F}( \| \eta\|_{Z^s(T)})\big).
\end{equation}
\end{proposition}

\begin{proof}
We apply $\langle T_{(1+\eta_x^2)^{-\frac{s}{2}}}D\rangle^s$ to the equation \eqref{ParaOneMuskat}, and denote $\eta_s := \langle T_{(1+\eta_x^2)^{-\frac{s}{2}}}D\rangle^s \eta$.
Then $\eta_s$ solves the equation
\begin{equation*} 
     \p_t \eta_s = -\frac{\sigma}{\mu^{-}} |D|T_\ell \eta_s -\frac{\sigma}{\mu^{-}}\big[\langle T_{(1+\eta_x^2)^{-\frac{s}{2}}}D \rangle^s, |D|T_\ell \big] \eta_s +  \langle T_{(1+\eta_x^2)^{-\frac{s}{2}}}D\rangle^sK.
\end{equation*}
Taking the time derivative of $\|\eta_s\|_{L^2}^2$, we get 
\begin{align*}
 \f12\frac{d}{dt} \|\eta_s\|_{L^2}^2 =& -\frac{\sigma}{\mu^{-}} (|D|T_\ell \eta_s, \eta_s)_{L^2 \times L^2} - \frac{\sigma}{\mu^{-}} ([\langle T_{(1+\eta_x^2)^{-\frac{s}{2}}}D \rangle^s, |D|T_\ell] \eta_s, \eta_s)_{L^2 \times L^2} \\
 &+ (\langle T_{(1+\eta_x^2)^{-\frac{s}{2}}}D\rangle^sK, \eta_s)_{L^2 \times L^2}.
\end{align*}
From the above proposition, for the last term involving $K$,
\begin{align*}
&\left|(\langle T_{(1+\eta_x^2)^{-\frac{s}{2}}}D\rangle^sK, \eta_s)_{L^2 \times L^2}\right|
\leq\|\langle T_{(1+\eta_x^2)^{-\frac{s}{2}}}D\rangle^sK, \eta_s\|_{H^{-\f52+\delta}}\|\eta_s\|_{H^{\f52-\delta}}\\
&\lesssim_{\|\eta\|_{H^s}}\left(\frac{\sigma}{\mu^{-}}   + \frac{\rho^- g}{\mu^-}\right)\|\eta\|_{H^{s+\f52}}\| \eta\|_{H^{s+\f52-\delta}}.
\end{align*}
Using the symbolic calculus \eqref{CompositionPara}, the commutator satisfies
\begin{equation*}
 \big\|[\langle T_{(1+\eta_x^2)^{-\frac{s}{2}}}D \rangle^s, |D|T_\ell] \eta_s \big\|_{H^{-\f52+\delta}} \lesssim_{\|\eta\|_{H^s}} \|\eta\|_{H^{s+\f52}},
\end{equation*}
so that the commutator term has the estimate
\begin{equation*}
\big| ([\langle T_{(1+\eta_x^2)^{-\frac{s}{2}}}D \rangle^s, |D|T_\ell] \eta_s, \eta_s)_{L^2 \times L^2}\big| \lesssim_{\|\eta\|_{H^s}}\|\eta\|_{H^{s+\f52}}\| \eta\|_{H^{s+\f52-\delta}}.
\end{equation*}
For the first term in the time derivative of $\|\eta_s\|_{L^2}^2$, we split the symbol as the sum of the leading order term, and the lower-order term 
\begin{equation*}
    |D|T_\ell =  T_{m^{[5]}} + T_{m^{[\leq 4]}}, \quad m^{[5]} = (1+\eta_x^2)^{-\f52}|\xi|^5,
\end{equation*}
where the symbol $m^{[\leq 4]}$ is the lower-order part of $|D|T_\ell$ of order at most $4$, and $m^{[5]}$ is strictly elliptic as long as $\eta \in \dot{W}^{1,\infty}$.
For the estimate that involves $m^{[\leq 4]}$, we get 
\begin{equation*}
 \frac{\sigma}{\mu^{-}} |(T_{m^{[\leq 4]}} \eta_s, \eta_s)_{L^2 \times L^2}| \leq \frac{\sigma}{\mu^{-}} \| T_{m^{[\leq 4]}} \eta_s\|_{H^{-\f52}} \| \eta_s\|_{H^{\f52}}  \lesssim_{\|\eta\|_{H^s}} \frac{\sigma}{\mu^{-}}\|\eta\|_{H^{s+\f52}}\| \eta\|_{H^{s+\f52-\delta}}.     
\end{equation*}
For the term that involves $m^{[5]}$, we write 
\begin{align*}
 (T_{m^{[5]}} \eta_s, \eta_s)_{L^2 \times L^2} &= \big(T_{\sqrt{m^{[5]}}} \eta_s, T_{\sqrt{m^{[5]}}}\eta_s\big)_{L^2 \times L^2} + \big(T_{\sqrt{m^{[5]}}} \eta_s, \big((T_{\sqrt{m^{[5]}}})^{*} -T_{\sqrt{m^{[5]}}}\big)\eta_s\big)_{L^2 \times L^2} \\
 &+ \big((T_{m^{[5]}}-T_{\sqrt{m^{[5]}}} T_{\sqrt{m^{[5]}}}) \eta_s, \eta_s\big)_{L^2 \times L^2}.
\end{align*}
For the first term in $(T_{m^{[5]}} \eta_s, \eta_s)_{L^2 \times L^2}$, following the estimate $(3.18)$ in \cite{MR4131404}, by using \eqref{TABound} and \eqref{CompositionPara}, we have
\begin{align*}
\|\eta_s\|_{H^{\f52}}&\leq
\|T_{\sqrt{m^{[5]}}^{-1}}T_{\sqrt{m^{[5]}}}\eta_s\|_{H^{\f52}}
+\|(Id-T_{\sqrt{m^{[5]}}^{-1}}T_{\sqrt{m^{[5]}}})\eta_s\|_{H^{\f52}}\\
&\lesssim_{\|\eta\|_{H^s}}\|T_{\sqrt{m^{[5]}}}\eta_s\|_{L^2}+\|\eta_s\|_{H^{\f52-\delta}},
\end{align*}
which gives that
\begin{equation*}
  \|T_{\sqrt{m^{[5]}}} \eta_s\|^2_{L^2} \geq \frac{1}{\mathcal{F}(\|\eta\|_{H^s})} \|\eta\|^2_{H^{s+\f52}} - \mathcal{F}(\|\eta\|_{H^s})\|\eta\|_{H^{s+\f52}}\| \eta\|_{H^{s+\f52-\delta}},
\end{equation*}
for some function $\mathcal{F}: \R^+ \rightarrow \R^+$.
Using the symbolic calculus \eqref{CompositionPara} and \eqref{AdjointBound}, 
\begin{equation*}
 \|\big((T_{\sqrt{m^{[5]}}})^{*} -T_{\sqrt{m^{[5]}}}\big)\eta_s \|_{H^\delta} + \|(T_{m^{[5]}}-T_{\sqrt{m^{[5]}}} T_{\sqrt{m^{[5]}}}) \eta_s\|_{H^{-\f52 + \delta}} \lesssim_{\|\eta\|_{H^s}}\|\eta\|_{H^{s+\f52}}.
\end{equation*}
we bound the second and the third terms in $(T_{m^{[5]}} \eta_s, \eta_s)_{L^2 \times L^2}$,
\begin{align*}
    &\big(T_{\sqrt{m^{[5]}}} \eta_s, \big((T_{\sqrt{m^{[5]}}})^{*} -T_{\sqrt{m^{[5]}}}\big)\eta_s\big)_{L^2 \times L^2} + \big((T_{m^{[5]}}-T_{\sqrt{m^{[5]}}} T_{\sqrt{m^{[5]}}}) \eta_s, \eta_s\big)_{L^2 \times L^2}\\
    &\lesssim_{\|\eta\|_{H^s}}\|\eta\|_{H^{s+\f52}}\| \eta\|_{H^{s+\f52-\delta}}.
\end{align*}
Combining the above estimates, we obtain
\begin{equation*}
-(|D|T_\ell \eta_s, \eta_s)_{L^2 \times L^2} \leq - \frac{1}{\mathcal{F}(\|\eta\|_{H^s})} \|\eta\|^2_{H^{s+\f52}} +\mathcal{F}(\|\eta\|_{H^s})\|\eta\|_{H^{s+\f52}}\| \eta\|_{H^{s+\f52-\delta}}. 
\end{equation*}
As a consequence,
\begin{equation*}
 \f12\frac{d}{dt} \|\eta_s\|_{L^2}^2 \leq -\frac{\sigma}{\mu^{-}}\frac{1}{\mathcal{F}(\|\eta\|_{H^s})} \|\eta\|^2_{H^{s+\f52}} + \Big(\frac{\sigma}{\mu^{-}} + \frac{\rho^- g}{\mu^-}  \Big) \mathcal{F}(\|\eta\|_{H^s})\|\eta\|_{H^{s+\f52}}\| \eta\|_{H^{s+\f52-\delta}},
\end{equation*}
for some function $\mathcal{F}$ that depends only on $h$ and $s$.
For the factor $\| \eta\|_{H^{s+\f52-\delta}}$, we use the interpolation
\begin{equation*}
 \| \eta\|_{H^{s+\f52-\delta}} \lesssim \| \eta\|_{H^s}^{1-\theta}\| \eta\|_{H^{s+\f52}}^{\theta},
\end{equation*}
for some constant $\theta \in (0,1)$ depending on $\delta$ and $s$.
Using Young's inequality,  we then get
\begin{align*}
\f12\f{d}{dt}\|\eta_s\|^2_{L^2}
\leq-\f{1}{\mathcal{F}(\|\eta\|_{H^s})}\|\eta\|^2_{H^{s+\f52}}+\mathcal{F}(\|\eta\|_{H^s})\|\eta\|^2_{H^{s}},
\end{align*}
which implies that for $t\in(0,T]$
\begin{align*}
\f{d}{dt}\|\eta_s\|^2_{L^2}
\leq-\f{1}{\mathcal{F}\big(\|\eta\|_{Z^s(T)} \big)}\|\eta\|^2_{H^{s+\f52}}+\mathcal{F}\big(\|\eta\|_{Z^s(T)}\big)\|\eta\|^2_{H^{s}}.
\end{align*}
Hence, taking the time integral, for $t\in(0,T]$
\begin{align*}
\|\eta_s(t,\cdot)\|^2_{L^2}+\f{1}{\mathcal{F}\big(\|\eta\|_{Z^s(T)} \big)}\int_0^t\|\eta(s,\cdot)\|^2_{H^{s+\f52}}ds\leq \|\eta_s(0,\cdot)\|^2_{L^2}+T\mathcal{F}\big(\|\eta\|_{Z^s(T)}\big)\|\eta\|^2_{L^\infty_t H^{s}_x}.
\end{align*}
Since we have the norm equivalence
\begin{align*}
\f{1}{\mathcal{F}\big(\|\eta\|_{Z^s(T)} \big)}\|\eta(t,\cdot)\|_{H^s}\leq\|\eta_s(t,\cdot)\|_{L^2}\leq\mathcal{F}\big(\|\eta\|_{Z^s(T)}\big)\|\eta(t,\cdot)\|_{H^s},
\end{align*}
we then obtain that for $t\in(0,T]$,
\begin{align*}
\|\eta(t,\cdot)\|^2_{H^s}+\int_0^t\|\eta(s,\cdot)\|^2_{H^{s+\f52}}ds\leq \mathcal{F}\left(\|\eta\|_{Z^s(T)}\right)\left(\|\eta(0, \cdot)\|^2_{H^s}+T\mathcal{F}\big(\|\eta\|_{Z^s(T)}\big)\|\eta\|^2_{Z^s(T)}\right),
\end{align*}
which can be written as 
\begin{align*}
\|\eta\|_{Z^s(T)}^2\leq \mathcal{F}\left(\|\eta\|_{Z^s(T)}\right)\left(\|\eta(0, \cdot)\|_{H^s}^2+T\mathcal{F}\big(\|\eta\|_{Z^s(T)}\big)\|\eta\|^2_{Z^s(T)}\right).
\end{align*}
This leads to the a priori estimate \eqref{OneAprioriEst}.
\end{proof}

We also need an a priori estimate for the distance between the surface $\Sigma_t$ and the bottom $\Gamma^{-}$.
\begin{lemma} \label{t:DistBound}
Let $s\geq 2$, and $\eta \in Z^s(T)$ is a solution to \eqref{ParaOneMuskat} that satisfies \eqref{BottomH}.
Then there exist $\theta \in (0,1)$ and a function $\mathcal{F}: \R^+ \rightarrow \R^+$ depending only on $(h,s,\frac{\rho^- g}{\mu^-},\frac{\sigma}{\mu^-})$ such that
\begin{equation} \label{DistBound}
    \inf_{t\in[0,T]} \text{dist}(\eta(t), \Gamma^{-}) \geq \text{dist}(\eta(0), \Gamma^{-}) - T^\theta \mathcal{F}(\| \eta\|_{Z^s(T)}).
\end{equation}
In particular, for $T>0$ sufficiently small, the free surface $\Sigma_t$ remains uniformly separated from the rigid boundary $\Gamma^-$.
\end{lemma}

\begin{proof}
Using the equation \eqref{OneMuskat}, the estimate \eqref{GfEtaf} and the fact that $s+\f52\geq\f{9}{2}$, we estimate
 \begin{align*}
 &\|\eta(t) - \eta(0)\|_{H^{-\f12}}\leq \int_0^t \frac{\sigma}{\mu^{-}} \|G^{-}(\eta)\mathbf{E}(\eta)(\tau)\|_{H^{-\f12}} + \frac{\rho^{-}g}{\mu^{-}} \|G^{-}(\eta)\eta(\tau)\|_{H^{-\f12}} d\tau \\
 \lesssim&  \int_0^t \mathcal{F}(\|\eta\|_{H^s}) \| \eta\|_{H^\f92}d\tau \leq t^\f12 \mathcal{F}\big(\|\eta\|_{L^\infty([0,t];H^s)}\big)\|\eta\|_{L^2([0,t];H^{s+\f52})}.
 \end{align*}
 For any constant $\nu\in(\f12,s)$, we use interpolation of Sobolev spaces:
\begin{equation*}
 \|\eta(t) - \eta(0)\|_{H^\nu} \leq \|\eta(t) - \eta(0)\|_{H^{-\f12}}^\theta \|\eta(t) - \eta(0)\|_{H^{s}}^{1-\theta} \leq t^{\frac{\theta}{2}} \mathcal{F}\big(\|\eta\|_{Z^s(t)}\big),
\end{equation*}
for some constant $\theta \in (0,1)$ that depends on $s$ and $\nu$.
Using the Sobolev embedding $H^\nu(\R) \subset L^\infty(\R)$, we obtain the estimate \eqref{DistBound}.
\end{proof}

\subsection{The two-phase case}
For the two-phase problem, we first recall a well-posedness result of \eqref{TwoMuskatTwo} from Proposition~4.1 in \cite{MR4131404}.
\begin{lemma} [\hspace{1sp}\cite{MR4131404}] \label{t:fpmHHalfEst}
Let $\eta \in W^{1,\infty}(\R)\cap H^\f12(\R)$ satisfy dist$(\eta, \Gamma^\pm) >h>0$.
Then there exists a unique variational solution $f^\pm \in \tilde{H}^\f12_\pm (\R)$ to the system \eqref{TwoMuskatTwo}.
Moreover, $f^\pm$ satisfy
\begin{equation*} 
 \|f^\pm\|_{\tilde{H}^\f12_\pm} \leq C(1+\|\eta\|_{W^{1,\infty}})^2 \| \sigma \mathbf{E}(\eta) + g(\rho^- - \rho^+)\eta\|_{H^\f12},
\end{equation*}
where the constant $C$ depends only on $(h, \mu^\pm)$.
\end{lemma}
Using the estimate \eqref{ellEst}, the above estimate is reduced to
\begin{equation} \label{fpmEst}
 \|f^\pm\|_{\tilde{H}^\f12_\pm} \lesssim_{\|\eta\|_{W^{1,\infty}}}  \sigma \|\eta\|_{H^{\f92}} + g(\rho^- - \rho^+)\|\eta\|_{H^\f12}.
\end{equation}

For the profile $\eta$ with higher Sobolev regularity, we change the estimate for the capillary term to the elastic term in Proposition $4.2$ in \cite{MR4131404}, and obtain the following result.
\begin{lemma} [\hspace{1sp}\cite{MR4131404}]
Let $f^\pm$ be the solution given by Lemma \ref{t:fpmHHalfEst}.
If $\eta \in H^{s+\f52}$ with $s\geq 2$ then $f^\pm \in \tilde{H}^{s-\f12}_{\pm}(\R)$, and 
\begin{equation} \label{fpmEstR}
 \|f^\pm\|_{\tilde{H}^r_\pm} \lesssim_{\|\eta\|_{H^s}}  \sigma \|\eta\|_{H^{r+4}} + g(\rho^- - \rho^+)\|\eta\|_{H^r},
\end{equation}
for all $r\in [\f12, s-\f12]$, where the implicit constant depends only on $(h,s,r,\mu^\pm)$.
\end{lemma}

Then we rewrite the two-phase  Hele-Shaw flow with an elastic interface \eqref{TwoMuskatOne}-\eqref{TwoMuskatTwo} in the paradifferential form.
\begin{proposition}
 For $\delta \in (0,\f12]$, the two-phase problem can be written as
 \begin{equation} \label{ParaTwoMuskat}
    \p_t \eta = -\frac{\sigma}{\mu^+ +\mu^{-}} |D|T_\ell \eta + K,
 \end{equation}
 where the remainder term $K$ satisfies the estimate
\begin{equation*}
    \|K\|_{H^{s-\f52+\delta}} \lesssim_{\|\eta\|_{H^s}} \sigma  \|\eta\|_{H^{s+\f52}} + (\rho^- - \rho^+)g\| \eta\|_{H^{s-\f32+\delta}}.
\end{equation*}
\end{proposition}

\begin{proof}
According to the estimate \eqref{GfDRf}, for $\delta \in (0,\f12]$, and $\mathfrak{s} \in [\f12, s-\delta]$, there exists some function $\mathcal{F}:\R^+ \rightarrow \R^+$ depending only on $(s,\mathfrak{s}, \delta, h)$ such that
\begin{equation*} 
G^\pm(\eta)f^\pm = \mp|D|f^\pm + R^\pm(\eta)f^\pm, \quad \| R^\pm(\eta)f^\pm\|_{H^{\mathfrak{s}- 1+\delta}} \leq \mathcal{F}(\|\eta\|_{H^s}) \|f\|_{\tilde{H}^\mathfrak{s}_-}.
\end{equation*}
 We then write the system \eqref{TwoMuskatTwo} as 
\begin{align*}
 |D|f^- &= \frac{\mu^-}{\mu^+ + \mu^-} |D|(\sigma \mathbf{E}(\eta)+ (\rho^- - \rho^+)g\eta) + \frac{\mu^-}{\mu^+ + \mu^-}R^+(\eta)f^+ - \frac{\mu^+}{\mu^+ + \mu^-}R^-(\eta)f^- \\
 &= \frac{\sigma\mu^-}{\mu^+ + \mu^-} |D|T_\ell \eta + \frac{\sigma\mu^-}{\mu^+ + \mu^-} |D|(\mathbf{E}(\eta)-T_\ell \eta) + \frac{\mu^-  (\rho^- - \rho^+)g}{\mu^+ + \mu^-} |D|\eta\\
 &+ \frac{\mu^-}{\mu^+ + \mu^-}R^+(\eta)f^+ - \frac{\mu^+}{\mu^+ + \mu^-}R^-(\eta)f^-.
\end{align*}
We then estimate using \eqref{ellEst} and \eqref{fpmEstR},
\begin{align*}
 &\| |D|(\mathbf{E}(\eta)-T_\ell \eta)\|_{H^{s-\f52 + \delta}} \lesssim_{\|\eta\|_{H^s}} \|\eta\|_{H^{s+\f32}}, \quad \||D|\eta \|_{H^{s-\f52 + \delta}} \lesssim_{\|\eta\|_{H^s}} \|\eta\|_{H^{s-\f32+\delta}}, \\
 & \|R^\pm(\eta) f^\pm\|_{H^{s-\f52+\delta}}\lesssim_{\|\eta\|_{H^s}}\|f^\pm\|_{\tilde{H}_\pm^{s-\f32}} \lesssim_{\|\eta\|_{H^s}}  \sigma \|\eta\|_{H^{s+\f52}} + g(\rho^- - \rho^+)\|\eta\|_{H^{s-\f32}}.
\end{align*}
Hence, we obtain that
\begin{equation*}
|D|f^- = \frac{\sigma\mu^-}{\mu^+ + \mu^-} |D|T_\ell \eta  + K.
\end{equation*}
Using the equation \eqref{TwoMuskatOne}, we get
\begin{equation*}
\p_t \eta = -\frac{1}{\mu^-}|D|f^- -\frac{1}{\mu^-}R^-(\eta)f^- = -\frac{\sigma}{\mu^+ +\mu^{-}} |D|T_\ell \eta + K,
\end{equation*}
which gives the reformulation of the two-phase  problem \eqref{ParaTwoMuskat}.
\end{proof}

Similar to the one-phase problem, as long as $\eta \in \dot{W}^{1,\infty}$, the two-phase problem \eqref{ParaTwoMuskat} is a fifth-order quasilinear parabolic equation.
Since the paradifferential form of the two-phase problem \eqref{ParaTwoMuskat} is similar to the paradifferential form of the one-phase problem \eqref{ParaOneMuskat}, we follow the computations in Proposition~\ref{t:OnePhaseEst} and Lemma~\ref{t:DistBound}, and obtain the a priori estimate for the two-phase problem.
\begin{proposition} \label{t:TwoPhaseEst}
Let $s\geq 2$, and suppose that $\eta$ solves the paradifferential two-phase problem \eqref{ParaTwoMuskat} on $[0,T]$, and satisfies the separation condition \eqref{BottomH}.
Then there exist $\theta\in (0,1)$ depending only on $s$, and a function $\mathcal{F}: \R^+ \rightarrow \R^+$ depending only on $(h,s, \sigma, \mu^\pm, (\rho^- -\rho^+)g)$ such that
\begin{equation*} 
  \| \eta\|_{Z^s(T)} \lesssim \mathcal{F}\big(\|\eta(0,\cdot)\|_{H^s} + T^\f12\mathcal{F}( \| \eta\|_{Z^s(T)})\big).
\end{equation*}
Moreover, 
\begin{equation*} 
    \inf_{t\in[0,T]} \text{dist}(\eta(t), \Gamma^{-}) \geq \text{dist}(\eta(0), \Gamma^{-}) - T^\theta \mathcal{F}\big(\| \eta\|_{Z^s(T)}\big).
\end{equation*}
\end{proposition}

\section{Contraction estimates and proof of local well-posedness} \label{s:Contraction}
This section is dedicated to proving the contraction estimates for both the one-phase and two-phase problems.
These results are then combined with the a priori estimates from Section \ref{s:Reduct} to establish Theorem \ref{t:MainOne}.

We begin with a contraction estimate for the remainder term in the paralinearization of the elastic term $R_\mathbf{E}$ in \eqref{ElasticPara}.

\begin{lemma} \label{t:REEst}
For $\delta\in (0,s- \f32)$, and $\delta \leq 1$, $R_\mathbf{E}$ satisfies the contraction estimate
\begin{equation*}
 \|R_E(\eta_1) - R_E(\eta_2) \|_{H^{s-\f32}} \lesssim_{\|(\eta_1, \eta_2)\|_{H^s \times H^s}}  \Big(\|\eta_1\|_{H^{s+\f52}}+ \|\eta_2\|_{H^{s+\f52}} \Big) \| \eta_1 -\eta_2\|_{H^s}.
\end{equation*}
\end{lemma}
\begin{proof}
We denote the G\^{a}teaux derivative $d_uF(u)$ of a function $F$ at $u$ in the direction $\dot{u}$ as follows:
\begin{align*}
d_uF(u)\dot{u}=\lim_{\e\to0}\f{1}{\e}(F(u+\e\dot{u})-F(u)).
\end{align*}
To obtain this lemma, it suffices to prove
\begin{align} \label{dEtaREEta}
\|d_\eta R_{E}(\eta)\dot{\eta}\|_{H^{s-\f32}}\lesssim_{\|\eta\|_{H^s}}  \|\eta\|_{H^{s+\f52}} \| \dot{\eta}\|_{H^s}.
\end{align}
One can check that
\begin{align*}
&d_{\eta}\left(\f{\eta_x}{\sqrt{1+\eta_x^2}}\right)\dot{\eta}=\f{\dot{\eta}_x}{(1+\eta_x^2)^{\f32}},\quad
d_{\eta}\left(\f{1}{1+\eta_x^2}\right)\dot{\eta}=-\f{2\eta_x\dot{\eta}_x}{(1+\eta_x^2)^2},\\
&d_{\eta}\left(\f{\eta_x}{(1+\eta_x^2)^{\f72}}\right)\dot{\eta}=\f{1-6\eta_x^2}{(1+\eta_x^2)^{\f92}}\dot{\eta}_x.
\end{align*}
Hence, using the chain rule for the G\^{a}teaux derivative,
\begin{align*}
d_{\eta}\left[\f{1}{1+\eta_x^2}\left(\dfrac{\eta_x}{\sqrt{1+\eta_x^2}}\right)_x \right]\dot{\eta}
=\f{\dot{\eta}_{xx}}{(1+\eta^2_x)^{\f52}}-\f{5\eta_x\eta_{xx}\dot{\eta}_x}{(1+\eta_x^2)^{\f72}},\\
d_{\eta}\left(\f{\eta_x\eta_{xx}^2}{(1+\eta^2_x)^{\f72}}\right)\dot{\eta}
=\f{(1-6\eta_x^2)\eta^2_{xx}}{(1+\eta_x^2)^{\f92}}\dot{\eta}_x
+\f{2\eta_x\eta_{xx}}{(1+\eta_x^2)^{\f72}}\dot{\eta}_{xx}.
\end{align*}
Therefore, we have
\begin{align*}
&d_{\eta}\left(\f{1}{1+\eta_x^2}\left(\dfrac{\eta_x}{\sqrt{1+\eta_x^2}}\right)_x\right)_{xx}\dot{\eta}=\f{\p_x^4\dot{\eta}}{(1+\eta^2_x)^{\f52}}-\f{10\eta_x\eta_{xx}\p^3_x\dot{\eta}}{(1+\eta_x^2)^{\f72}}\\
+&((1+\eta^2_x)^{-\f52})_{xx}\dot{\eta}_{xx}-\f{5\eta_x\eta_{xx}\dot{\eta}_{xxx}}{(1+\eta_x^2)^{\f72}}-\left(\f{10\eta_x\eta_{xx}}{(1+\eta_x^2)^{\f72}}\right)_x\dot{\eta}_{xx}-\left(\f{5\eta_x\eta_{xx}}{(1+\eta_x^2)^{\f72}}\right)_{xx}\dot{\eta}_{x},
\end{align*}
and
\begin{align*}
\f52d_{\eta}\left(\f{\eta_x\eta_{xx}^2}{(1+\eta^2_x)^{\f72}}\right)_x\dot{\eta}
=&\f{5\eta_x\eta_{xx}}{(1+\eta_x^2)^{\f72}}\p^3_{x}\dot{\eta}+\left(\f{5\eta_x\eta_{xx}}{(1+\eta_x^2)^{\f72}}\right)_x\dot{\eta}_{xx}\\
+&\f{5(1-6\eta_x^2)\eta^2_{xx}}{2(1+\eta_x^2)^{\f92}}\dot{\eta}_{xx}+\left(\f{5(1-6\eta_x^2)\eta^2_{xx}}{2(1+\eta_x^2)^{\f92}}\right)_x\dot{\eta}_{x}.
\end{align*}
Thus, we obtain that
\begin{align*}
&d_{\eta}\mathbf{E}(\eta)\dot{\eta}=
\f{\p_x^4\dot{\eta}}{(1+\eta^2_x)^{\f52}}+2\big((1+\eta^2_x)^{-\f52}\big)_x\p^3_x\dot{\eta}\\
&+\left\{((1+\eta^2_x)^{-\f52})_{xx}-\left(\f{5\eta_x\eta_{xx}}{(1+\eta_x^2)^{\f72}}\right)_x+\f{5(1-6\eta_x^2)\eta^2_{xx}}{2(1+\eta_x^2)^{\f92}}\right\}\dot{\eta}_{xx}\\
&-\left[\left(\f{5\eta_x\eta_{xx}}{(1+\eta_x^2)^{\f72}}\right)_{xx}-\left(\f{5(1-6\eta_x^2)\eta^2_{xx}}{2(1+\eta_x^2)^{\f92}}\right)_x\right]\dot{\eta}_{x}.
\end{align*}
Recall that
\begin{align*}
   T_{\ell}\dot{\eta} =&T_{(1+\eta_x^2)^{-\f52}} \p_x^4\dot{\eta} + 2T_{((1+\eta_x^2)^{-\f52})_x} \p_x^3\dot{\eta}+ T_{((1+\eta_x^2)^{-\f52})_{xx}-5(\eta_{xx}\eta_x(1+\eta_x^2)^{-\f72})_x+\f52\eta_{xx}^2(1-6\eta_x^2)(1+\eta_x^2)^{-\f92}} \dot{\eta}_{xx}\\
+&T_{-5(\eta_{xx}\eta_x(1+\eta_x^2)^{-\f72})_{xx}+\f52(\eta_{xx}^2(1-6\eta_x^2)(1+\eta_x^2)^{-\f92})_x} \dot{\eta}_x.
\end{align*}
By using \eqref{BalancedError}, we then get
\begin{align*}
\|d_{\eta}\mathbf{E}(\eta)\dot{\eta}-T_{\ell}\dot{\eta}\|_{H^{s-\f32}}\lesssim_{\|\eta\|_{H^s}}   \|\eta\|_{H^{s+\f52}} \| \dot{\eta}\|_{H^s}.
\end{align*}
Since 
\begin{align*}
d_{\eta}R_\mathbf{E}\dot{\eta}=d_{\eta}\mathbf{E}(\eta)\dot{\eta}-T_{\ell}\dot{\eta}-T_{d_{\eta}\ell\dot{\eta}}\eta,
\end{align*}
it remains to prove
\begin{equation} \label{tdEtaEllEtaEta}
\|T_{d_{\eta}\ell\dot{\eta}}\eta\|_{H^{s-\f32}}\lesssim_{\|\eta\|_{H^s}} \|\eta\|_{H^{s+\f52}} \| \dot{\eta}\|_{H^s}.
\end{equation}
This can be deduced from \eqref{TABound}. Thus, we prove \eqref{dEtaREEta}, which leads to the proof of the lemma.
\end{proof}
\begin{lemma}
For $0<\delta \leq \f12$ and $s\geq2$, $\mathbf{E}(\eta)$ satisfies the contraction estimate
\begin{equation} \label{ElasticDiff}
\begin{aligned}
 \|\mathbf{E}(\eta_1) - \mathbf{E}(\eta_2) \|_{H^{s-\f32-\delta}} &\lesssim_{\|(\eta_1,\eta_2)\|_{H^s}}\|\eta_1 - \eta_2\|_{H^{s+\f52-\delta}}\\
 +&\left(\|\eta_1\|_{H^{s+\f52-\delta}}+\|\eta_2\|_{H^{s+\f52-\delta}}\right)\|\eta_1 - \eta_2\|_{H^{s}}.
 \end{aligned}
\end{equation}
\end{lemma}
\begin{proof}
Recall that
\begin{equation*}
d_{\eta}\mathbf{E}(\eta)\dot{\eta} = d_{\eta}R_\mathbf{E}\dot{\eta}+T_{\ell}\dot{\eta}+T_{d_{\eta}\ell\dot{\eta}}\eta.
\end{equation*}
Then we use \eqref{ellEst}, \eqref{dEtaREEta} and \eqref{tdEtaEllEtaEta} to estimate
\begin{align*}
 \|d_{\eta}\mathbf{E}(\eta)\dot{\eta} \|_{H^{s-\f32-\delta}} &\leq \|d_{\eta}R_\mathbf{E}\dot{\eta} \|_{H^{s-\f32-\delta}} + \|T_{\ell}\dot{\eta}\|_{H^{s-\f32-\delta}} + \|T_{d_{\eta}\ell\dot{\eta}}\eta\|_{H^{s-\f32-\delta}} \\
 &\lesssim_{\|\eta\|_{H^s}} \|\eta\|_{H^{s+\f52}} \| \dot{\eta}\|_{H^s} + \| \eta\|_{H^s}\|\dot{\eta}\|_{H^{s+\f52-\delta}}.
\end{align*}
This leads to the estimate \eqref{ElasticDiff}.
\end{proof}

\subsection{Contraction estimate for one-phase problem}
If $\eta_1$ and $\eta_2$ are two solutions of the one-phase problem \eqref{OneMuskat} in $Z^s(T)$ satisfying \eqref{BottomH}, then we show that the $Z^s(T)$ norm of $\eta_1 - \eta_2$ can be controlled by the initial condition of $\eta_1 - \eta_2$ in $H^s$.
\begin{proposition}\label{t:propConOne}
Let $s>2$.
If $\eta_1$ and $\eta_2$ are two solutions of the one-phase problem \eqref{OneMuskat} in $Z^s$ satisfying \eqref{BottomH}, then
\begin{equation} \label{ContractOneMuskat}
 \|\eta_1 - \eta_2\|_{Z^s(T)} \lesssim_{\|(\eta_1, \eta_2)\|_{Z^s(T) \times Z^s(T)}} \|(\eta_1 - \eta_2)(0,\cdot) \|_{H^s}.
\end{equation}
The implicit constant in the inequality depends only on $\big(s,h, \frac{\sigma}{\mu^-},  \frac{\rho^- g}{ \mu^-}\big)$.
\end{proposition}

\begin{proof}
Write $\delta\eta = \eta_1 - \eta_2$ for the difference of two solutions to the one-phase problem.
From the equation \eqref{OneMuskat}, one can write
\begin{equation} \label{DiffMuskatOne}
    \p_t \delta\eta = -\f{\sigma}{\mu^-}G^-(\eta_1)(\mathbf{E}(\eta_1) - \mathbf{E}(\eta_2))- \mathcal{R}_0,
\end{equation}
where the remainder term $\mathcal{R}_0$ is given by
\begin{equation*}
 \mathcal{R}_0 : = \f{\rho^- g}{\mu^-}G^-(\eta_1) \delta\eta + [G^-(\eta_1)-G^-(\eta_2)]\cdot \left( \f{\sigma}{\mu^-}\mathbf{E}(\eta_2) + \f{\rho^- g}{\mu^-}\eta_2\right).
\end{equation*}
Using the estimate \eqref{GfEtaf},  
\begin{equation*}
 \|G^-(\eta_1) \delta\eta \|_{H^{s-\f52}} \lesssim_{\|\eta_1\|_{H^s}} \|\delta \eta\|_{H^{s-\f32}}.
\end{equation*}
According to the difference estimate for the Dirichlet-Neumann operator \eqref{GetaDiff},  we estimate
\begin{align*}
& \left\|[G^-(\eta_1)-G^-(\eta_2)]\cdot\left(\f{\sigma}{\mu^-}\mathbf{E}(\eta_2)+\f{\rho^- g}{\mu^-}\eta_2\right)\right\|_{H^{s-\f52}} \\
 &\lesssim_{\|(\eta_1, \eta_2)\|_{H^s \times H^s}} \left(  \f{\sigma}{\mu^-} +  \f{\rho^- g}{\mu^-}\right) \|\delta \eta\|_{H^s} \|\eta_2\|_{H^{s+\f52}}
\end{align*}
Next, we rewrite the term $G^-(\eta_1)(\mathbf{E}(\eta_1) - \mathbf{E}(\eta_2))$.
Using the paralinearization of the Dirichlet-Neumann operator \eqref{GfDRf},
\begin{align*}
 G^-(\eta_1)(\mathbf{E}(\eta_1) - \mathbf{E}(\eta_2)) &= |D|(\mathbf{E}(\eta_1) - \mathbf{E}(\eta_2)) + \mathcal{R}_1, \\ & \|\mathcal{R}_1\|_{H^{s-\f52}}\lesssim_{\|\eta_1\|_{H^s}} \|\mathbf{E}(\eta_1) - \mathbf{E}(\eta_2) \|_{H^{s-\f32 - \delta}}.
\end{align*}
For the difference estimate of two elastic terms, we recall the estimate in \eqref{ElasticDiff}, 
\begin{equation*}
\|\mathbf{E}(\eta_1) - \mathbf{E}(\eta_2) \|_{H^{s-\f32 - \delta}} \lesssim_{\|(\eta_1, \eta_2)\|_{H^s \times H^s}} \|\delta\eta\|_{H^{s+\f52 -\delta}} +\Big(\|\eta_1\|_{H^{s+\f52-\delta}}+\|\eta_2\|_{H^{s+\f52-\delta}}\Big)\|\delta\eta\|_{H^{s}}.
\end{equation*}
Hence, we get the estimate for $\mathcal{R}_1$:
\begin{equation*}
\|\mathcal{R}_1\|_{H^{s-\f52}}\lesssim_{\|(\eta_1, \eta_2)\|_{H^s \times H^s}} \|\delta\eta\|_{H^{s+\f52 -\delta}}.
\end{equation*}
We then write 
\begin{equation*}
 |D|(\mathbf{E}(\eta_1) - \mathbf{E}(\eta_2)) = |D|T_{\ell_1}\delta\eta + |D|T_{\ell_1 - \ell_2}\eta_2 + |D|(R_E(\eta_1)-R_E(\eta_2)).
\end{equation*}
By the result in Lemma \ref{t:REEst}, 
\begin{equation*}
 \||D|(R_E(\eta_1)-R_E(\eta_2))\|_{H^{s-\f52}} \lesssim_{\|(\eta_1, \eta_2)\|_{H^s \times H^s}}  \Big(\|\eta_1\|_{H^{s+\f52}}+ \|\eta_2\|_{H^{s+\f52}} \Big) \|\delta \eta\|_{H^s}.
\end{equation*}
From \eqref{tdEtaEllEtaEta}, we get that
\begin{equation*}
 \| |D|T_{\ell_1 - \ell_2}\eta_2\|_{H^{s-\f52}} \lesssim_{\|(\eta_1, \eta_2)\|_{H^s \times H^s}}  \| \delta \eta\|_{H^s}\|\eta_2\|_{H^{s+\f52}}.
\end{equation*}
Collecting the above estimates, the equation \eqref{DiffMuskatOne} can be written as
\begin{equation} \label{ParaDiffMuskatOne}
\p_t\delta \eta=-\f{\sigma}{\mu^-}|D|T_{\ell_1}\delta\eta+\mathcal{R}_2,
\end{equation}
where the remainder term $\mathcal{R}_2$ satisfies
\begin{equation*}
\|\mathcal{R}_2\|_{H^{s-\f52}}\lesssim_{\|(\eta_1, \eta_2)\|_{H^s \times H^s}} \|\delta\eta\|_{H^{s+\f52-\delta}} +\Big(\|\eta_1\|_{H^{s+\f52}}+\|\eta_2\|_{H^{s+\f52}}\Big)\| \delta \eta\|_{H^s}.
\end{equation*}
Note that the equation \eqref{ParaDiffMuskatOne} has the same form as the equation \eqref{ParaOneMuskat}.
Hence, doing the same analysis as in the proof of Proposition \ref{t:OnePhaseEst}, we get the energy inequality
\begin{align*}
\f12\frac{d}{dt}\|\delta\eta\|^2_{H^s} &\leq - \f{1}{\mathcal{F}(\|(\eta_1, \eta_2)\|_{H^s\times H^s})} \|\delta \eta \|^2_{H^{s+\f52}} + \mathcal{F}(\|(\eta_1, \eta_2)\|_{H^s\times H^s}) \|\delta \eta\|_{H^{s+\f52}} \|\delta \eta\|_{H^{s+\f52-\delta}} \\
&+ \mathcal{F}(\|(\eta_1, \eta_2)\|_{H^s\times H^s})\|(\eta_1, \eta_2)\|_{H^{s+\f52}\times H^{s+\f52}}  \|\delta \eta\|_{H^s} \|\delta \eta\|_{H^{s+\f52}},
\end{align*}
for some function $\mathcal{F}: \R^+ \rightarrow \R^+$ depending only on $(h,s,\frac{\rho^- g}{\mu^-},\frac{\sigma}{\mu^-})$.
By interpolation and Young's inequality, we bound the last two terms in the inequality by 
\begin{align*}
&\mathcal{F}(\|(\eta_1, \eta_2)\|_{H^s\times H^s}) \|\delta \eta\|_{H^{s+\f52}} \|\delta \eta\|_{H^{s+\f52-\delta}} \\
\leq& \f{1}{4\mathcal{F}(\|(\eta_1, \eta_2)\|_{H^s\times H^s})} \|\delta \eta \|^2_{H^{s+\f52}} + \mathcal{F}_1(\|(\eta_1, \eta_2)\|_{H^s\times H^s}) \|\delta \eta\|^2_{H^s}, \\
& \mathcal{F}(\|(\eta_1, \eta_2)\|_{H^s\times H^s}) \Big(\|(\eta_1, \eta_2)\|_{H^{s+\f52}\times H^{s+\f52}}\Big)  \|\delta \eta\|_{H^s} \|\delta \eta\|_{H^{s+\f52}} \\
\leq& \f{1}{4\mathcal{F}(\|(\eta_1, \eta_2)\|_{H^s\times H^s})} \|\delta \eta \|^2_{H^{s+\f52}} + \mathcal{F}_1(\|(\eta_1, \eta_2)\|_{H^s\times H^s}) \|(\eta_1, \eta_2)\|_{H^{s+\f52}\times H^{s+\f52}}^2 \|\delta \eta\|^2_{H^s},
\end{align*}
for some functions $\mathcal{F}, \mathcal{F}_1$ and $\mathcal{F}_2$ depending only on $(h,s,\frac{\rho^- g}{\mu^-},\frac{\sigma}{\mu^-})$.
Hence, we get the inequality 
\begin{align*}
\f12\frac{d}{dt}\|\delta\eta\|^2_{H^s} &\leq - \f{1}{\mathcal{F}(\|(\eta_1, \eta_2)\|_{H^s\times H^s})} \|\delta \eta \|^2_{H^{s+\f52}} \\ &\quad+\mathcal{F}(\|(\eta_1, \eta_2)\|_{H^s\times H^s}) \|(\eta_1, \eta_2)\|_{H^{s+\f52}\times H^{s+\f52}}^2  \|\delta \eta\|^2_{H^s}.
\end{align*}
Applying Gronwall's inequality and using the definition of the $Z^s(T)$ norm, we obtain the desired estimate \eqref{ContractOneMuskat}.
\end{proof}

\subsection{Contraction estimate for two-phase  problem}
If $\eta_1$ and $\eta_2$ are two solutions of the two-phase problem \eqref{TwoMuskatOne}-\eqref{TwoMuskatTwo} in $Z^s(T)$ satisfying \eqref{BottomH}, then similar to the one-phase case, we show that the $Z^s(T)$ norm of $\eta_1 - \eta_2$ can be controlled by the initial condition of $\eta_1 - \eta_2$ in $H^s$.

For $j = 1,2$, consider $f_j^\pm$ that solve the system 
\begin{equation*} 
\begin{cases}
f^-_j - f^+_j =  \sigma \mathbf{E}(\eta_j) + g(\rho^- - \rho^+)\eta_j, \\
\frac{1}{\mu^+}G^+(\eta_j)f^+_j = \frac{1}{\mu^-}G^-(\eta_j)f_j^-.
\end{cases}
\end{equation*}
We write $\delta f^\pm = f^\pm_1 - f^\pm_2$, and $\delta \eta = \eta_1 - \eta_2$ for the difference of these functions.
We first prove the following contraction estimate for $f^\pm$.

\begin{lemma}
Let $s\geq 2$, and $\delta \in (0, \f12]$.
For each $r\in [\f12, s-\f32]$,
\begin{equation} \label{DeltaFpmEst}
\begin{aligned}
 \|\delta f^\pm \|_{\tilde{H}^r_\pm} &\lesssim_{\|(\eta_1, \eta_2)\|_{H^s\times H^s}} \sigma \|\delta \eta \|_{H^{r+4}}+ (\rho^- -\rho^+)g \|\delta \eta\|_{H^r}\\
 +& (\sigma+(\rho^--\rho^+)g)\| \delta \eta\|_{H^s}\|(\eta_1, \eta_2)\|_{H^{s+\f52}\times H^{s+\f52}}.
\end{aligned}    
\end{equation}
The implicit constant in the inequality depends only on $\big(s,h, r,\mu^\pm\big)$.
\end{lemma}

\begin{proof}
We write
\begin{align*}
&\delta f^- - \delta f^+ = \delta k = \sigma (\mathbf{E}(\eta_1)-\mathbf{E}(\eta_2)) + g(\rho^- - \rho^+)\delta \eta, \\
&\f{1}{\mu^-}G^-(\eta_1)\delta f^- - \f{1}{\mu^+}G^+(\eta_1)\delta f^+ = \f{1}{\mu^+}[G^+(\eta_1)-G^+(\eta_2)]f_2^+ - \f{1}{\mu^-}[G^-(\eta_1)-G^-(\eta_2)]f_2^-.
\end{align*}
Applying $G^+(\eta_1)$ to the first equation, and using the fact that $G^\pm(\eta_1)\delta f^\pm = \mp |D|\delta f^\pm + R^\pm(\eta_1)\delta f^\pm$, we get
\begin{align*}
&|D|\delta f^- = \f{\mu^-}{\mu^+ + \mu^-}(\sigma|D|[\mathbf{E}(\eta_1)-\mathbf{E}(\eta_2)] + g(\rho^- - \rho^+)|D|\delta \eta) + \f{\mu^+ \mu^-}{\mu^++\mu^-}F, \\
F =& \f{1}{\mu^+}R^+(\eta_1)\delta f^+ -  \f{1}{\mu^-}R^-(\eta_1)\delta f^- + \f{1}{\mu^+}[G^+(\eta_1)-G^+(\eta_2)]f^+_2  - \f{1}{\mu^-}[G^-(\eta_1)-G^-(\eta_2)]f_2^-.
\end{align*}
Using the contraction estimate \eqref{ElasticDiff}, 
\begin{align*}
 \|\delta f^-\|_{H^r} &\lesssim_{\|(\eta_1, \eta_2)\|_{H^{s}\times H^{s}}} \sigma \|\delta \eta \|_{H^{r+4}}+ (\rho^- -\rho^+)g \|\delta \eta\|_{H^r} \\ &\quad+\|F\|_{H^{r-1}} + \sigma \| (\eta_1,\eta_2) \|_{H^{r+4}\times H^{r+4}}\| \delta \eta\|_{H^{s}} .
\end{align*}
Using estimates \eqref{GfDRf}, \eqref{GetaDiff} and \eqref{fpmEstR},
\begin{align*} 
&\| R^\pm(\eta_1)\delta f^\pm\|_{H^{r-1}} \lesssim_{\|(\eta_1,\eta_2)\|_{H^s\times H^s}}\|\delta f^{\pm}\|_{\tilde{H}^{r-\delta}_{\pm}}, \\
 &\|[G^{\pm}(\eta_1) - G^{\pm}(\eta_2)]f_2^{\pm}\|_{H^{s-\f52}} \lesssim_{\|(\eta_1, \eta_2)\|_{H^s \times H^s}} (\sigma + (\rho^- -\rho^+)g)\|\delta \eta\|_{H^s} \|(\eta_1, \eta_2)\|_{H^{s+\f52}\times H^{s+\f52}}.
\end{align*}
For $r\in[\f12,s-\f32]$, we then have
\begin{align*}
\|F\|_{H^{r-1}}
\lesssim_{\|(\eta_1,\eta_2)\|_{H^s\times H^s}}\|\delta f^{\pm}\|_{\tilde{H}^{r-\delta}_{\pm}}
+(\sigma + (\rho^- -\rho^+)g)\|\delta \eta\|_{H^s} \|(\eta_1, \eta_2)\|_{H^{s+\f52}\times H^{s+\f52}}.
\end{align*}
Note that $\delta f^+=\delta f^--\delta k$, we then obtain that
\begin{align*}
\|\delta f^{\pm}\|_{H^r}  \lesssim_{\|(\eta_1, \eta_2)\|_{H^{s}\times H^{s}}}& \|\delta f^{\pm}\|_{\tilde{H}^{r-\delta}_{\pm}}+\sigma \|\delta \eta \|_{H^{r+4}}+ (\rho^- -\rho^+)g \|\delta \eta\|_{H^r}\\
&+ (\sigma + (\rho^- -\rho^+)g)\| (\eta_1,\eta_2) \|_{H^{s+\f52}\times H^{s+\f52}}\| \delta \eta\|_{H^s} .
\end{align*}
From the definition of $\tilde{H}^{r}$, following the proof of (4.22) in \cite{MR4131404}, we can obtain
\begin{align*}
&\|\delta f^{\pm}\|_{H^{1,r}}\lesssim_{\|(\eta_1,\eta_2)\|_{H^s\times H^s}}\|\delta f^{\pm}\|_{\tilde{H}^{\f12}_{\pm}}+\|\delta f^{\pm}\|_{H^{1,r-\delta}}\\
&+\sigma \|\delta \eta \|_{H^{r+4}}+ (\rho^- -\rho^+)g \|\delta \eta\|_{H^r}+ (\sigma + (\rho^- -\rho^+)g)\| (\eta_1,\eta_2) \|_{H^{s+\f52}\times H^{s+\f52}}\| \delta \eta\|_{H^s}.
\end{align*}

Using the variational formulation of the two-phase Muskat problem in \cite{MR4090462} and the definition of $\tilde{H}^{\f12}_{\pm}$, we have
\begin{align*}
&\|\delta f^{\pm}\|_{\tilde{H}^{\f12}_{\pm}}  \lesssim_{\|(\eta_1, \eta_2)\|_{H^{s}\times H^{s}}} \|\mathbf{E}(\eta_1)-\mathbf{E}(\eta_2)\|_{H^{\f12}}+ \|\delta\eta\|_{H^s}(\|\mathbf{E}(\eta_1)\|_{H^{\f12}}+\|\mathbf{E}(\eta_2)\|_{H^{\f12}})\\
\lesssim&_{\|(\eta_1, \eta_2)\|_{H^{s}\times H^{s}}}\sigma\|\delta\eta\|_{H^{\f92}}+(\rho^--\rho^+)g\|\delta\eta\|_{H^{\f12}}\\ &\quad+(\sigma+(\rho^--\rho^+)g)\|\delta\eta\|_{H^s}\left(\|\eta_1\|_{H^{\f92}}+\|\eta_2\|_{H^{\f92}}\right).
\end{align*}
From the above estimates, by using an induction argument, we  obtain \eqref{DeltaFpmEst}.
\end{proof}

We then prove the contraction estimate for two solutions $\eta_1$ and $\eta_2$.
\begin{proposition} \label{t:propConTwo}
Let $s\geq 2$.
If $\eta_1$ and $\eta_2$ are two solutions of the two-phase problem \eqref{TwoMuskatOne}-\eqref{TwoMuskatTwo} in $Z^s$ satisfying \eqref{BottomH}, then
\begin{equation*} %\label{ContractTwoMuskat}
 \|\eta_1 - \eta_2\|_{Z^s(T)} \lesssim_{\|(\eta_1, \eta_2)\|_{Z^s(T) \times Z^s(T)}} \|(\eta_1 - \eta_2)(0,\cdot) \|_{H^s}.
\end{equation*}
The implicit constant in the inequality depends only on $\big(s,h, \sigma, \mu^\pm, (\rho^- -\rho^+)g\big)$.
\end{proposition}

\begin{proof}
From the equation \eqref{TwoMuskatOne}, $\delta \eta$ solves the equation
\begin{equation*}
\p_t \delta\eta = -\f{1}{\mu^-}G^-(\eta_1)\delta f^- - \f{1}{\mu^-}[G^-(\eta_1) - G^-(\eta_2)]f_2^-.
\end{equation*}
Using estimates \eqref{GetaDiff} and \eqref{fpmEstR},
\begin{equation*}
 \|[G^-(\eta_1) - G^-(\eta_2)]f_2^- \|_{H^{s-\f52}} \lesssim_{\|(\eta_1, \eta_2)\|_{H^s \times H^s}} (\sigma + (\rho^- -\rho^+)g)\|\delta \eta\|_{H^s} \|(\eta_1, \eta_2)\|_{H^{s+\f52}\times H^{s+\f52}}.
\end{equation*}
Using the paralinearization of the Dirichlet-Neumann operator \eqref{GfDRf}, we write
\begin{equation*}
    G^-(\eta_1) \delta f^- = |D|\delta f^- + R^-(\eta_1)\delta f^-,
\end{equation*}
where the remainder term satisfies
\begin{align*}
 &\|R^-(\eta_1)\delta f^- \|_{H^{s-\f52}}\lesssim_{\|(\eta_1, \eta_2)\|_{H^s \times H^s}}\|\delta f^-\|_{\tilde{H}^{s-\f32-\delta}}\\
 &\lesssim_{\|(\eta_1, \eta_2)\|_{H^s\times H^s}} \sigma \|\delta \eta \|_{H^{s+\f52-\delta}}+ (\rho^- -\rho^+)g \|\delta \eta\|_{H^{s-\f32-\delta}}\\
 &+ (\sigma + (\rho^- -\rho^+)g)\|\delta \eta\|_{H^s} \|(\eta_1, \eta_2)\|_{H^{s+\f52}\times H^{s+\f52}},
\end{align*}
where we have used \eqref{DeltaFpmEst} with $r=s-\f32-\delta$.
We then get that
\begin{align*}
\p_t\delta\eta=-\f{1}{\mu^-}|D|\delta f^-+\mathcal{R}_1,
\end{align*}
where the remainder term satisfies
\begin{align*}
\|\mathcal{R}_1\|_{H^{s-\f52}}\lesssim_{\|(\eta_1, \eta_2)\|_{H^s\times H^s}}(\sigma+(\rho^--\rho^+)g) \left(\|\delta \eta \|_{H^{s+\f52-\delta}}
 + \| \delta \eta\|_{H^s}\|(\eta_1, \eta_2)\|_{H^{s+\f52}\times H^{s+\f52}}\right).
\end{align*}
From the computations in the preceding lemma, 
\begin{align*}
&|D|\delta f^- = \f{\mu^-}{\mu^+ + \mu^-}(\sigma|D|[\mathbf{E}(\eta_1)-\mathbf{E}(\eta_2)] + g(\rho^- - \rho^+)|D|\delta \eta) + \f{\mu^+ \mu^-}{\mu^++\mu^-}F,
\end{align*}
where the remainder term satisfies
\begin{align*}
\|F\|_{H^{s-\f52}}\lesssim_{\|(\eta_1, \eta_2)\|_{H^s\times H^s}}(\sigma+(\rho^--\rho^+)g) \left(\|\delta \eta \|_{H^{s+\f52-\delta}}
 + \| \delta \eta\|_{H^s}\|(\eta_1, \eta_2)\|_{H^{s+\f52}\times H^{s+\f52}}\right).
\end{align*}
From the calculations of \eqref{ParaDiffMuskatOne}, we have
\begin{align*}
\||D|[\mathbf{E}(\eta_1)-\mathbf{E}(\eta_2)]-|D|T_{\ell_1}\delta\eta\|_{H^{s-\f52}}
\lesssim_{\|(\eta_1, \eta_2)\|_{H^s\times H^s}}\|\delta \eta \|_{H^{s+\f52-\delta}}
+\|\delta\eta\|_{H^s}\|(\eta_1, \eta_2)\|_{H^{s+\f52}\times H^{s+\f52}}.
\end{align*}
Thus, we obtain that
\begin{equation} \label{ParaDiffMuskatTwo}
\p_t\delta \eta=-\f{\sigma}{\mu^++\mu^-}|D|T_{\ell_1}\delta\eta+\mathcal{R}_2,
\end{equation}
where the remainder term $\mathcal{R}_2$ satisfies
\begin{align*}
\|\mathcal{R}_2\|_{H^{s-\f52}}\lesssim_{\|(\eta_1, \eta_2)\|_{H^s\times H^s}}(\sigma+(\rho^--\rho^+)g) \left(\|\delta \eta \|_{H^{s+\f52-\delta}}
 + \| \delta \eta\|_{H^s}\|(\eta_1, \eta_2)\|_{H^{s+\f52}\times H^{s+\f52}}\right).
\end{align*}
Since \eqref{ParaDiffMuskatTwo} takes the same form as \eqref{ParaDiffMuskatOne}, the result follows from an argument analogous to that in Proposition \ref{t:propConOne}.
\end{proof}

\subsection{Proof of local well-posedness}
We now finish the proof of local well-posedness by adapting the a priori estimates and contraction estimates we have shown.
The argument here is essentially the same as in Sections~3.4 and~4.4 of \cite{MR4131404}.
We assume that $\eta_0 \in H^s(\R)$ for $s\geq 2$ satisfying dist$(\eta_0, \Gamma^\pm)>2h$. 

We begin with the one-phase problem.
Let $J_\varepsilon$ be the mollifier that selects the frequency portion not greater than $\varepsilon^{-1}$ of a function.
Then we consider the approximate solution $\eta_\varepsilon$ that solves the  following ODE in the Banach space 
\begin{equation*}
    \p_t \eta_\varepsilon = -\f{1}{\mu^-} J_\varepsilon [G^-(J_\varepsilon \eta_\varepsilon)(\sigma \mathbf{E}(J_\varepsilon\eta_\varepsilon)+ \rho^-gJ_\varepsilon\eta_\varepsilon)], \quad  \eta_\varepsilon|_{t = 0} = \eta(0, x).
\end{equation*}
The solution $\eta_\varepsilon$ exists on some maximal time interval $(0, T_\varepsilon]$. 
The a priori estimates in Proposition~\ref{t:OnePhaseEst} and Lemma~\ref{t:DistBound} also hold for $\eta_\varepsilon$.
Hence, using a continuity argument, there exists a positive time $T<T_\varepsilon$ for all $\varepsilon \in (0,1]$ such that
\begin{equation*}
  \| \eta_\varepsilon\|_{Z^s(T)} \lesssim \mathcal{F}\big(\|\eta(0,\cdot)\|_{H^s}\big), \quad  \inf_{t\in[0,T]} \text{dist}(\eta_\varepsilon(t), \Gamma^{-}) >h,
\end{equation*}
for some function $\mathcal{F}$ that depends on $\big(s,h, \frac{\sigma}{\mu^-},  \frac{\rho^- g}{ \mu^-}\big)$.
The contraction estimate in Proposition~\ref{t:propConOne} also holds for $\eta_\varepsilon$.
By letting $\varepsilon \rightarrow 0$, we obtain that the limiting function $\eta$ is a solution of \eqref{OneMuskat} in $Z^s(T)$ with initial data $\eta_0$.
Uniqueness of the solution and Lipschitz dependence on the initial data follow directly from Proposition \ref{t:propConOne}.

For the two-phase problem, we consider the ODE
\begin{equation*}
    \p_t \eta_\varepsilon = -\f{1}{\mu^-} J_\varepsilon [G^-(J_\varepsilon \eta_\varepsilon)(J_\varepsilon f_\varepsilon^-)], \quad  \eta_\varepsilon|_{t = 0} = \eta(0, x),
\end{equation*}
where functions $f^\pm_\varepsilon$ solve
\begin{equation*}  
\begin{cases}
f^-_\varepsilon - f^+_\varepsilon = \sigma \mathbf{E}(\eta_\varepsilon) + g(\rho^- - \rho^+)\eta_\varepsilon, \\
\frac{1}{\mu^+}G^+(\eta_\varepsilon)f^+_\varepsilon = \frac{1}{\mu^-}G^-(\eta_\varepsilon)f^-_\varepsilon.
\end{cases}
\end{equation*}
The solvability and regularity of $f^\pm$ follow from Lemma \ref{t:fpmHHalfEst}.
The a priori estimates in Proposition~\ref{t:TwoPhaseEst} and the contraction estimate in Proposition~\ref{t:propConTwo} also hold for $\eta_\varepsilon$.
By passing to the limit $\varepsilon \rightarrow 0$, and doing the same as in the one-phase case, we obtain the local well-posedness of two-phase  problem.

\section{Global well-posedness of Hele-Shaw flow with an elastic interface} \label{s:Global}
In this section, we prove Theorem \ref{t:global}, establishing the global well-posedness for the Hele-Shaw flow with an elastic interface under the assumption of small initial data.
This section transitions from the local energy-based arguments to an integral formulation, utilizing fixed-point lemmas to prove that solutions exist for all time $T>0$.

The overarching strategy in this section is to treat the nonlinear system as a perturbation of a fractional heat equation.
By rewriting the equations in integral form, we can apply space-time estimates of the fractional heat equation to show that the nonlinearities remain small if the initial data is sufficiently small.

\subsection{Global well-posedness for one-phase problem} \label{s:GlobalOneMuskat}
For the one-phase problem, we rewrite the governing equation \eqref{OneMuskat} as:
\begin{equation*} 
\begin{aligned}
&\left(\p_t + \f{\sigma}{\mu^-}|D|^5 + \frac{\rho^- g}{\mu^-}|D| \right) \eta\\
=& -\f{\sigma}{\mu^-}R^-(\eta)\mathbf{E}(\eta) - \f{\sigma}{\mu^-} |D|(T_{\ell(\eta)}\eta+R_{\mathbf{E}}\eta - |D|^4\eta)  -\f{\rho^- g}{\mu^-} R^-(\eta) \eta.
\end{aligned}
\end{equation*}
Hence, $\eta$ is the solution to the integral equation
\begin{equation} \label{OneMuskatIntegral}
    \eta(t) = e^{- \f{\sigma}{\mu^-}t|D_x|^5 - \frac{\rho^- g}{\mu^-}t|D_x|} \eta_0 - \mathcal{B}(\eta, \eta),
\end{equation}
where the bilinear form is given by
\begin{equation*} 
\begin{aligned}
\mathcal{B}(\eta, f)(t,x) = \f{\sigma}{\mu^-}\int_0^t& e^{- \f{\sigma}{\mu^-}(t- \tau)|D_x|^5 - \frac{\rho^- g}{\mu^-}(t-\tau)|D_x|}\Big( R^-(\eta)(T_{\ell(\eta)} f + R_{\mathbf{E}} (f))\\
+&  |D|(T_{\ell(\eta)}f+R_{\mathbf{E}}f - |D|^4f)  +\f{\rho^- g}{\sigma} R^-(\eta) f \Big) d\tau.  
\end{aligned}
\end{equation*}

Let $X^s([0, T])$ denote the Banach space endowed with the equivalent norm
\begin{equation*}
 \|\eta\|_{X^s([0, T])}=\| \eta\|_{\wt L^\infty([0, T]; H^s)}+\f{\sigma}{\mu^-}\| \eta\|_{\wt L^1([0, T]; H^{s+5})}.
\end{equation*}
We prove the following result for the equation \eqref{OneMuskatIntegral}.
\begin{lemma} \label{t:OneMuskatFixed}
Let $s>\f32$. There exists a small number $\delta>0$ such that if  $\| \eta_0\|_{H^s}<\delta$, then 
\begin{equation*}
    e^{- \f{\sigma}{\mu^-}t|D_x|^5 - \frac{\rho^- g}{\mu^-}t|D_x|} \eta_0 - \mathcal{B}(\eta, \eta),
\end{equation*} 
 has a unique fixed point $\eta$ in $X^s([0, T])$ for any $T>0$ with norm less than $C\| \eta_0\|_{H^s}$.
\end{lemma}
\begin{proof}
Let $T>0$. We apply the fixed-point Lemma \ref{t:fixedpoint}. 
In the proof the constant $C$ may change from line to line.
From \eqref{SpaceTimeOne}, we have
\begin{align*}
\| e^{- \f{\sigma}{\mu^-}t|D_x|^5 - \frac{\rho^- g}{\mu^-}t|D_x|} \eta_0\|_{X^s([0,T])}
\leq C\|\eta_0\|_{H^s}.
\end{align*}
To match the paralinearization of the elastic term in \eqref{ElasticPara}, we define 
\begin{equation*}
 \mathbf{E}(\eta, f) : = T_{\ell(\eta)} f + R_{\mathbf{E}} (f).
\end{equation*}
Using \eqref{SpaceTimeOne} with $(q_1,q_2)=(\infty,1)$ and $(q_1,q_2)=(1,1)$, we have
\begin{align*}
\|\mathcal{B}(\eta,f)\|_{X^s([0,T])}
\leq& C\f{\sigma}{\mu^-}\|R^-(\eta)\mathbf{E}(\eta,f) \|_{\tilde{L}^1([0,T];H^s)}+C\f{\rho^- g}{\mu^-} \|R^-(\eta) f \|_{\tilde{L}^1([0,T];H^s)}\\
&+C\f{\sigma}{\mu^-}\||D|(T_{\ell(\eta)}f+R_{\mathbf{E}}f - |D|^4f)  \|_{\tilde{L}^1([0,T];H^s)}.
\end{align*}
We now estimate each term on the right-hand side of the inequality.
It follows from Lemma \ref{t:elliptic} that
\begin{align*}
\|R^-(\eta) f \|_{H^s}\leq C\|\eta\|_{W^{1+\e,\infty}}\|f\|_{H^{s+1}}+C\|\eta\|_{H^{s+1}}\|f\|_{W^{1,\infty}}.
\end{align*}
We also get that
\begin{align*}
&\|R^-(\eta)\mathbf{E}(\eta,f) \|_{H^s}
\leq C\|\eta\|_{W^{1+\e,\infty}}\|\mathbf{E}(\eta,f)\|_{H^{s+1}}+C\|\eta\|_{H^{s+1}}\|\mathbf{E}(\eta,f)\|_{H^s}\\
&\leq C\|\eta\|_{H^s}\|f\|_{H^{s+5}}+C\|f\|_{H^s}\|\eta\|_{H^{s+5}}
+C\|\eta\|_{H^{s+1}}\|f\|_{H^{s+4}}\\
&\leq C\|\eta\|_{H^s}\|f\|_{H^{s+5}}+C\|f\|_{H^s}\|\eta\|_{H^{s+5}}.
\end{align*}
Due to \eqref{ellEst}, we have 
\begin{align*}
\||D|(T_{\ell(\eta)}f+R_{\mathbf{E}}f - |D|^4f)\|_{H^s}\leq C\|\eta\|_{H^s}\|f\|_{H^{s+5}}+C\|f\|_{H^s}\|\eta\|_{H^{s+5}}.
\end{align*}
We then obtain that 
\begin{equation} \label{check:cB1}
\begin{aligned}
\|\mathcal{B}(\eta,f)\|_{X^s([0,T])}\leq&C\|\eta\|_{L^{\infty}([0,T];H^s)}\|f\|_{\tilde{L}^{1}([0,T];H^{s+5})}+C\|f\|_{L^{\infty}([0,T];H^s)}\|\eta\|_{\tilde{L}^{1}([0,T];H^{s+5})}\\
\leq&C\|\eta\|_{X^s([0,T])}\|f\|_{X^s([0,T])}.
\end{aligned}  
\end{equation}
From the contraction estimate \eqref{GetaDiff},
we can obtain
\begin{align*}
&\|[R^-(\eta_1)-R^-(\eta_2)]\mathbf{E}(\eta_2,f) \|_{H^s}\\
\leq& C\|\delta\eta\|_{H^s}[(\|\eta_1\|_{H^s}+\|\eta_2\|_{H^s})\|f\|_{H^{s+5}}+(\|\eta_1\|_{H^{s+5}}+\|\eta_2\|_{H^{s+5}})\|f\|_{H^{s}}],
\end{align*}
and by Lemma \ref{t:elliptic},
\begin{align*}
&\|R^-(\eta_1)(\mathbf{E}(\eta_1,f) -\mathbf{E}(\eta_2,f))\|_{H^s}\\
\leq& C(\|\eta_1\|_{H^s}+\|\eta_2\|_{H^s})\|\delta \eta\|_{H^{s+5}}\|f\|_{H^s}+C\|\delta\eta\|_{H^{s+5}}\|f\|_{H^{s}}.
\end{align*}
Similarly, we have
\begin{align*}
&\||D|(T_{\ell(\eta_1)}f+R_{E(\eta_1)}f - |D|^4f)-|D|(T_{\ell(\eta_2)}f+R_{E(\eta_2)}f - |D|^4f)\|_{H^s} \\
\leq& C\|\delta\eta\|_{H^s}[(\|\eta_1\|_{H^s}+\|\eta_2\|_{H^s})\|f\|_{H^{s+5}}+(\|\eta_1\|_{H^{s+5}}+\|\eta_2\|_{H^{s+5}})\|f\|_{H^{s}}].
\end{align*}
Thus, we get the difference estimate
\begin{equation} \label{check:cB2}
\begin{aligned}
\|\mathcal{B}(\eta_1,f)-\mathcal{B}(\eta_2,f)\|_{X^s([0,T])}\leq&(\|(\eta_1,\eta_2)\|_{X^s\times X^s([0,T])}+1)\|\delta \eta\|_{X^s([0,T])}\|f\|_{X^s([0,T])}.
\end{aligned}
\end{equation}
We then obtain this lemma by using the fixed-point Lemma \ref{t:fixedpoint}.
\end{proof}
Let $\eta_0^j\in H^s$, $j=1, 2$ be initial data with norm less than $\delta$ given in Lemma \ref{t:OneMuskatFixed}, and denote by $\eta^j$, $j=1,2$, the corresponding solutions with initial data $\eta_0^j$.
Using \eqref{OneMuskatIntegral}, \eqref{check:cB1} and \eqref{check:cB2}, we compute
\begin{align*}
 &\|\eta_1 - \eta_2\|_{X^s([0,T])} \leq \big\|e^{- \f{\sigma}{\mu^-}t|D_x|^5 - \frac{\rho^- g}{\mu^-}t|D_x|} (\eta_0^1 - \eta_0^2) \big\|_{X^s([0,T])} +  \|\mathcal{B}(\eta^1, \eta^1) - \mathcal{B}(\eta^2, \eta^2) \|_{X^s([0,T])} \\
 &\leq C\|\eta_0^1 - \eta_0^2 \|_{H^s} + \|\mathcal{B}(\eta^1, \eta^1- \eta^2) \|_{X^s([0,T])} +  \|\mathcal{B}(\eta^1, \eta^2) - \mathcal{B}(\eta^2, \eta^2) \|_{X^s([0,T])} \\
 &\leq C\|\eta_0^1 - \eta_0^2 \|_{H^s}   +C\|\eta^1\|_{X^s([0, T])}\| \eta^1-\eta^2\|_{X^s([0, T])}+C\| \eta^1-\eta^2\|_{X^s([0, T])}\| \eta^2\|_{X^s([0, T])}.
\end{align*}
Choosing $\delta$ small enough, we obtain that for any $T>0$,
\begin{equation*}
 \|\eta_1 - \eta_2\|_{X^s([0,T])} \lesssim \|\eta_0^1 - \eta_0^2 \|_{H^s}.
\end{equation*}
This shows the Lipschitz dependence on initial data of the solution map.
As a consequence, we obtain the global well-posedness for one-phase problem in $H^s$ in the sense of Hadamard.

\subsection{Global well-posedness for two-phase problem} \label{s:GlobalTwoMuskat}
For the two-phase problem, we restrict ourselves to the physically stable case in which the denser fluid lies below the less dense fluid $\rho^+\leq \rho^-$.
One can write the system \eqref{TwoMuskatTwo} as 
\begin{equation*}
 \frac{1}{\mu^-}G^-(\eta)f^- = \frac{1}{\mu^+}G^+(\eta)f^- -\frac{g(\rho^- - \rho^+)}{\mu^+}G^+(\eta)\eta - \f{\sigma}{\mu^+}G^+(\eta) \mathbf{E}(\eta).
\end{equation*}
This can be further written as
\begin{equation*}
|D_x|f^- = \f{1}{\mu^+ + \mu^-}(\mu^- R^+(\eta)-\mu^+ R^-(\eta))f^-  -\f{g(\rho^- - \rho^+)\mu^-}{\mu^+ + \mu^-}G^+ (\eta) \eta -\f{\sigma \mu^-}{\mu^+ + \mu^-}G^+ (\eta) \mathbf{E}(\eta).
\end{equation*}
Hence, $f^-$ is a fixed point of 
\begin{equation*} 
\begin{aligned}
\mathcal{K}(\eta)\varphi = &\f{1}{\mu^+ + \mu^-}(\mu^- |D_x|^{-1}R^+(\eta)-\mu^+ |D_x|^{-1}R^-(\eta))\varphi \\
-&\f{g(\rho^- - \rho^+)\mu^-}{\mu^+ + \mu^-}|D_x|^{-1}G^+ (\eta) \eta -\f{\sigma \mu^-}{\mu^+ + \mu^-}|D_x|^{-1}G^+ (\eta) \mathbf{E}(\eta). 
\end{aligned}
\end{equation*}

Using the equation \eqref{TwoMuskatOne}, we obtain that $\eta$ solves the nonlinear fractional heat equation.
\begin{equation*}  %\label{PDE:twophase}
\begin{aligned}
 &\p_t \eta + \frac{\sigma}{\mu^+ + \mu^-}|D_x|^5 \eta + \frac{g(\rho^- - \rho^+)}{\mu^+ + \mu^-}|D_x|\eta = -\frac{1}{\mu^+ + \mu^-}( R^+(\eta)+ R^-(\eta))f^- \\
 +& \frac{g(\rho^- - \rho^+)}{\mu^+ + \mu^-}R^+(\eta)\eta + \f{\sigma}{\mu^++\mu^-}R^+(\eta) \mathbf{E}(\eta) - \f{\sigma}{\mu^++\mu^-}|D_x|( \mathbf{E}(\eta) - \p_x^4\eta).
\end{aligned}
\end{equation*}

\begin{proposition}\label{Prop:f-}
Let $r>\f12$ and $\eta\in H^s\cap H^{r+5}$. There exists a small number $\delta>0$ such that if  $\| \eta\|_{H^s}<\delta$, then the mapping $\mathcal{K}(\eta)$ has a unique fixed point $f^-$ in $H^s\cap H^{r+1}$ and
\begin{equation} \label{fMinusBound}
\|f^{-}\|_{H^{r+1}}\leq Cg(\rho^--\rho^+)\|\eta\|_{H^{r+1}}
+C\sigma\|\eta\|_{H^{r+5}}.
\end{equation}
\end{proposition}
\begin{proof}
We apply Lemma \ref{t:fp} with $E_1=H^s$, $E_2=H^{r+1}$ and 
\begin{align*}
u_0=-\f{g(\rho^- - \rho^+)\mu^-}{\mu^+ + \mu^-}|D_x|^{-1}G^+ (\eta) \eta -\f{\sigma \mu^-}{\mu^+ + \mu^-}|D_x|^{-1}G^+ (\eta) \mathbf{E}(\eta).
\end{align*}
From \eqref{GfDRf2}, we have
\begin{align*}
\|u_0\|_{H^{r+1}}\leq&g(\rho^--\rho^+)\||D_x|\eta\|_{H^{r}}
+\sigma\||D_x|\mathbf{E}(\eta)\|_{H^{r}}\\
\leq&Cg(\rho^--\rho^+)\|\eta\|_{H^{r+1}}
+C\sigma\|\eta\|_{H^{r+5}}.
\end{align*}
It remains to prove that the mapping
\begin{align*}
\mathcal{K}_1(\eta)\varphi := \f{1}{\mu^+ + \mu^-}(\mu^- |D_x|^{-1}R^+(\eta)-\mu^+ |D_x|^{-1}R^-(\eta))\varphi
\end{align*}
satisfies the condition in Lemma \ref{t:fp}. 
We apply the estimate \eqref{GfDRf2},
\begin{align*}
\|\mathcal{K}_1\varphi\|_{H^{s}}\leq C\|\varphi\|_{H^{s}},\quad
\|\mathcal{K}_1\varphi\|_{H^{r+1}}\leq C\|\eta\|_{H^s}\|\varphi\|_{H^{r+1}}+C\|\eta\|_{H^{r+1}}\|\varphi\|_{H^s}.
\end{align*}
Hence, there exists a unique fixed point $f^-\in H^s\cap H^{r+1}$ for $\mathcal{K}(\eta)$, and
\begin{align*}
\|f^-\|_{H^{r+1}}\leq Cg(\rho^--\rho^+)\|\eta\|_{H^{r+1}}
+C\sigma\|\eta\|_{H^{r+5}}+C\|\eta\|_{H^s}\|f^-\|_{H^{r+1}}+C\|\eta\|_{H^{r+1}}\|f^-\|_{H^s}.
\end{align*}
Taking $r=s-1$, we obtain
\begin{align*}
\|f^-\|_{H^s}\leq Cg(\rho^--\rho^+)\|\eta\|_{H^s}
+C\sigma\|\eta\|_{H^{s+4}}.
\end{align*}
Thus, we obtain the estimate \eqref{fMinusBound}.
\end{proof}

Assume that $\|(\eta_1,\eta_2)\|_{H^s\times H^s}\leq\delta$, and $f^-_j$ is the fixed point of $\mathcal{K}(\eta_j)$. Then
\begin{align*}
\||D_x|\delta f^-\|_{H^{r}}&\lesssim
\|R^+(\eta_1)f^-_1-R^+(\eta_2)f^-_2\|_{H^r}+\|R^-(\eta_1)f^-_1-R^-(\eta_2)f^-_2\|_{H^r}\\
+&g(\rho^--\rho^+)\|G^+(\eta_1)\eta_1-G^+(\eta_2)\eta_2\|_{H^r}+\sigma\|G^+(\eta_1)\mathbf{E}(\eta_1)-G^+(\eta_2)\mathbf{E}(\eta_2)\|_{H^r}.
\end{align*}
Using \eqref{GetaDiff} and \eqref{GfDRf2}, we then obtain
\begin{equation}\label{Global-two-deltaf}
\begin{aligned}
\||D_x|\delta f^-\|_{H^{r}}\leq C(\sigma+g(\rho^--\rho^+))\left(\|(\eta_1, \eta_2)\|_{H^{r+5}\times H^{r+5}}\|\delta\eta\|_{H^s}+\|\delta\eta\|_{H^{r+5}}\right).
\end{aligned}   
\end{equation}

From Proposition \ref{Prop:f-} and \eqref{Global-two-deltaf}, performing an analysis similar to that in the proof of global well-posedness for one-phase problem, we obtain the global well-posedness for two-phase problem.

\appendix
\section{Paradifferential and product estimates} \label{s:Paradifferential}
Here, we recall the definition of function spaces and some of the paradifferential estimates that are used in  previous sections.
Some of these results can be found, for instance, in \cite{MR2418072, MR2768550, MR3260858}.

\begin{definition}
We recall the  Littlewood-Paley frequency decomposition, $ I = \sum_{k\in \mathbb{N}} P_k$, where, for each $k\geq1$, $P_k$ is a smooth Fourier multiplier localized at frequency $2^k$, and $P_0$ selects the low-frequency components $|\xi|\leq 1$.
\begin{enumerate}
\item Let $s\in \mathbb{R}$, and $p,q \in [1, \infty]$.
The non-homogeneous Besov space $B^s_{p,q}(\mathbb{R})$ is defined as the space of all tempered distributions $u$ such that
\begin{equation*}
\| u\|_{B^s_{p,q}} : = \left\|(2^{ks}\|P_k u \|_{L^p})_{k=0}^\infty \right\|_{l^q} < +\infty.
\end{equation*}
\item When $p = q = \infty$, Besov space $B^s_{\infty, \infty}$ coincides with the \textit{Zygmund space} $C^s_{*}$.
When $p = q =2$, the Besov space $B^s_{2,2}$ becomes the \textit{Sobolev space} $H^s$.
\item One has the Sobolev embedding, 
\begin{equation}
 H^{s+\frac{1}{2}}(\mathbb{R}) \hookrightarrow C^s_{*}(\mathbb{R}) \quad \forall s, \label{HsCsEmbed}
\end{equation}
the Sobolev space $H^{s+\frac{1}{2}}(\mathbb{R})$ embeds into the Zygmund space $C^s_{*}(\mathbb{R})$.
\item Let $k\in\mathbb{N}$. We denote by $W^{k,\infty}(\mathbb{R})$ the space of all functionss such that $\partial_x^j u \in L^\infty(\mathbb{R})$, $0\leq j \leq k$. 
For $\rho = k+ \gamma$ with $k\in \mathbb{N}$ and $\gamma \in (0,1)$, we denote by $W^{\rho,\infty}(\mathbb{R})$ the
space of all functions $u\in W^{k,\infty}(\mathbb{R})$ such that $\partial_x^k u$ is $\gamma$-H\"{o}lder continuous on $\mathbb{R}$. 
\item The Zygmund space $C^s_{*}(\mathbb{R})$ is just the H\"{o}lder space $W^{s, \infty}(\mathbb{R})$ when $s\in (0,\infty)\backslash \mathbb{N}$.
One has the embedding properties
\begin{align*}
  &C_{*}^s(\mathbb{R}) \hookrightarrow L^\infty(\mathbb{R}), \quad s>0; \qquad L^\infty(\mathbb{R}) \hookrightarrow C_{*}^s, \quad s<0;\\
  &C_{*}^{s_1}(\mathbb{R})\hookrightarrow C_{*}^{s_2}(\mathbb{R}), \quad H^{s_1}(\mathbb{R})\hookrightarrow H^{s_2}(\mathbb{R}), \qquad s_1>s_2.
\end{align*}
\end{enumerate}
\end{definition}

\begin{definition}
\begin{enumerate}
\item Let $\rho\in [0,\infty)$, $m\in \mathbb{R}$. 
$\Gamma^m_\rho(\mathbb{R})$ denotes the space of locally bounded functions $a(x, \xi)$ on $\mathbb{R}\times (\mathbb{R}\backslash \{0\})$, which are $C^\infty$ with respect to $\xi$ for $\xi \neq 0$ and such that for all $k \in \mathbb{N}$ and $\xi \neq 0$, the function $x\mapsto \partial_\xi^k a(x,\xi)$ belongs to $W^{\rho,\infty}(\mathbb{R})$ and there exists a constant $C_k$ with
\begin{equation*}
\forall |\xi|\geq \frac{1}{2}, \quad \|\partial_\xi^k a(\cdot,\xi) \|_{W^{\rho,\infty}} \leq C_k (1+ |\xi|)^{m-k}.
\end{equation*}
Let $a\in\Gamma^m_\rho$. We define the semi-norm
\begin{equation*}
M^m_{\rho}(a) = \sup_{k \leq \frac{3}{2}+\rho} \sup_{|\xi|\geq \frac{1}{2}} \|(1+ |\xi|)^{k-m}\partial_\xi^k a(\cdot,\xi)  \|_{W^{\rho,\infty}}.
\end{equation*}
\item Given $a\in \Gamma^m_\rho(\mathbb{R})$, let $C^\infty$ functions $\chi(\theta, \eta)$ and $\psi(\eta)$ be such that for some $0<\epsilon_1 < \epsilon_2<1$,
\begin{align*}
    &\chi(\theta, \eta) = 1,  \text{ if } |\theta| \leq \epsilon_1(1+ |\eta|), \qquad \chi(\theta, \eta) = 0,  \text{ if } |\theta| \geq \epsilon_2(1+ |\eta|),\\
    &\psi(\eta) = 0, \text{ if } |\eta|\leq \frac{1}{5}, \qquad \psi(\eta) =1, \text{ if } |\eta|\geq \frac{1}{4}.
\end{align*}
We define the paradifferential operator $T_a$ by
\begin{align*}
 \widehat{T_a u}(\xi) = \frac{1}{2\pi}\int \chi(\xi -\eta, \eta) \hat{a}(\xi-\eta, \eta)\psi(\eta)\hat{u}(\eta) d\eta,
\end{align*}
where $\hat{a}(\theta, \xi)$ is the Fourier transform of $a$ with respect to the variable $x$.

\item Let $m\in \mathbb{R}$, an operator  is said to be of order $m$ if, for all $s\in \mathbb{R}$, it is bounded from $H^s$ to $H^{s-m}$.
\end{enumerate}
\end{definition}

We recall the basic symbolic calculus for paradifferential operators in the following result.
\begin{lemma}[Symbolic calculus, \cite{MR2418072}]
Let $m\in \mathbb{R}$ and $\rho\in [0, +\infty)$.
\begin{enumerate}
\item If $a\in \Gamma^m_0$, then the paradifferential operator $T_a$ is of order $m$. Moreover, for all $s\in \mathbb{R}$, there exists a positive constant $K$ such that
\begin{equation}
\|T_a\|_{H^s\rightarrow H^{s-m}} \leq K M^m_0(a). \label{TABound}
\end{equation}
\item If $a\in \Gamma^m_\rho$, and $b\in \Gamma^{m^{'}}_\rho$ with $\rho>0$, then the operator $T_a T_b -T_{a\sharp b}$ is of order $m+m^{'}-\rho$, where the composition
\begin{equation*}
    a \sharp b := \sum_{\alpha<\rho} \frac{(-i)^\alpha}{\alpha !} \partial^\alpha_\xi a(x,\xi) \partial^\alpha_x b(x,\xi).
\end{equation*}
Moreover,  for all $s\in \mathbb{R}$, there exists a positive constant $K$ such that
\begin{equation}
  \|T_a T_b-T_{a\sharp b} \|_{H^s \rightarrow H^{s-m-m^{'}+\rho}} \leq K\left(M^m_\rho(a)M^{m^{'}}_0(b) + M^m_0(a)M^{m^{'}}_\rho(b)\right). \label{CompositionPara}
\end{equation}
\item Let $a\in \Gamma^m_\rho$ with $\rho > 0$. 
Denote by $(T_a)^{*}$ the adjoint operator of $T_a$ and by $\bar{a}$ the complex conjugate of $a$.
Then $(T_a)^{*} -T_{a^{*}}$ is of order $m - \rho$, where 
\begin{equation*}
    a^{*} =  \sum_{\alpha<\rho} \frac{1}{i^\alpha \alpha !} \partial^\alpha_\xi \partial^\alpha_x \bar{a}.
\end{equation*}
Moreover,  for all $s\in \mathbb{R}$, there exists a positive constant $K$ such that
\begin{equation}
 \|(T_a)^{*} -T_{a^{*}} \|_{H^s \rightarrow H^{s-m+\rho}} \leq KM^m_\rho(a). \label{AdjointBound}
\end{equation}
In particular, if $a$ is a function that is independent of $\xi$, then  $(T_a)^* = T_{\bar{a}}$.
\end{enumerate}
\end{lemma}

When $a$ is just a function, $T_a u$ becomes the low-high paraproduct.
Below, we record some estimates for products and paraproducts.

\begin{lemma}[\hspace{1sp}\cite{MR2768550}]
\begin{enumerate}
\item Let $s_0, s_1, s_2$ be such that $s_0 \leq s_2$ and $s_0< s_1 +s_2 -\frac{1}{2}$, then
\begin{equation}
    \|T_a u\|_{H^{s_0}(\mathbb{R})} \lesssim \| a\|_{H^{s_1}(\mathbb{R})} \|u \|_{H^{s_2}(\mathbb{R})}.
\end{equation}
If in addition to the conditions above, $s_1 +s_2 >0$, then
\begin{equation} \label{BalancedError}
\|au -T_a u \|_{H^{s_0}(\mathbb{R})} \lesssim \| u\|_{H^{s_1}(\mathbb{R})} \|a \|_{H^{s_2}(\mathbb{R})}.
\end{equation}
\item For $s>0$ and $u,v\in H^s\cap L^\infty$, $uv\in H^s$, and 
\begin{equation} \label{HsLinftyEst}
 \|uv\|_{H^s} \lesssim \|u\|_{L^\infty} \|v\|_{H^s} + \|v\|_{L^\infty} \|u\|_{H^s}.
\end{equation}
\end{enumerate}
\end{lemma}

For nonlinear functions, we record below the Moser estimate, the difference estimate and the paralinearization result.
\begin{lemma}[\hspace{1sp}\cite{MR2768550}]
\begin{enumerate}
\item (Moser) Let $F\in C^\infty(\mathbb{C}^N)$ be a smooth function satisfying $F(0) = 0$.
Then, there holds
\begin{equation} \label{MoserOne}
\|F(u) \|_{H^s} \lesssim_{\|u\|_{L^\infty}} \|u\|_{H^s},\quad s\geq 0.  
\end{equation}
\item  Let $F\in C^\infty(\mathbb{C}^N)$ be a smooth function satisfying $\nabla F(0) = 0$.
For any $U,V \in H^s(\R^d)^N \cap L^\infty(\R^d)^N$, 
\begin{equation} \label{FUVDiffEst}
\begin{aligned}
& \|F(U)-F(V)\|_{H^s} \\
 \lesssim&_{\|U\|_{L^\infty}, \|V\|_{L^\infty}} \Big(  \|U-V\|_{H^s} + \|U-V\|_{L^\infty} \sup_{\tau\in [0,1]} \|V+\tau(U-V) \|_{H^s}\Big).
 \end{aligned}
\end{equation}
\item (Paralinearization) Let $s, \rho>0$, and $F(u)$ be a smooth function of $u$, then for any $u\in H^s(\mathbb{R}^d)\cap C^\rho_{*}(\mathbb{R}^d)$,
\begin{equation} \label{Paralinearization}
    \|F(u)-F(0)-T_{F^{'}(u)}u\|_{H^{s+\rho}(\mathbb{R}^d)} \leq C(\|u\|_{L^\infty(\mathbb{R}^d)})\|u\|_{C^\rho_{*}(\mathbb{R}^d)}\|u\|_{H^s(\mathbb{R}^d)}.
\end{equation}
\end{enumerate}
\end{lemma}

\section{Results on the Dirichlet-Neumann operator and parabolic estimates} \label{s:DNOperator}
In this section, we recall some results on the Dirichlet-Neumann operator.
We consider the following screened fractional Sobolev space defined in Leoni and Tice \cite{MR3989147}:
\begin{equation*}
 \widetilde{H}^\f{1}{2}_\Upsilon (\R) = \left\{ f\in \mathcal{S}'(\R) \cap L^2_{loc}(\R) : \int_\R \int_{B_{\R}(0, \Upsilon(x))} \dfrac{|f(x+y ) - f(x)|^2}{|y|^2} dydx < \infty  \right\} / \R, 
\end{equation*}
where $\Upsilon : \R \rightarrow (0, \infty]$ is a given lower semi-continuous function.
For the lower domain $\Omega^-$, we choose
\begin{equation*}
\Upsilon(x) = \begin{cases}
\infty, \quad \text{if }  \Gamma^- = \emptyset, \\
\mathfrak{d}_- (x): = \frac{\eta(x) - \underline{b}^-(x)}{2(\| \eta_x\|_{L^\infty} + \| \underline{b}^-_x\|_{L^\infty})}, \quad \text{if } \underline{b}^- \in \dot{W}^{1,\infty}(\R). 
\end{cases}
\end{equation*}
For the upper domain $\Omega^+$, $\underline{b}^-$ is replaced by $\underline{b}^+$. 
We define the space
\begin{equation*}
 \widetilde{H}^\f{1}{2}_\pm (\R) = \begin{cases}
\widetilde{H}^\f{1}{2}_\infty (\R), \quad \text{if }  \Gamma^- = \emptyset, \\
\widetilde{H}^\f{1}{2}_{\mathfrak{d}_\pm}(\R), \quad \text{if } \underline{b}^- \in \dot{W}^{1,\infty}(\R). 
\end{cases}
\end{equation*}

Let $\mathcal{S}'(\R)$ and $\mathcal{P}(\R)$ denote the space of tempered distributions on $\R$ and the set of polynomials on $\R$.
We also define the slightly-homogeneous Sobolev spaces
\begin{equation*}
H^{1,\sigma}(\R) = \{f \in \mathcal{S}'(\R)\cap L^2_{loc}(\R): f_x \in H^{\sigma-1}(\R) \}/ \R .
\end{equation*}
For $s>\f12$, we set 
\begin{equation*}
   \widetilde{H}^s_\pm (\R) = \widetilde{H}^\f{1}{2}_\pm(\R)\cap H^{1,s}(\R). 
\end{equation*}

Using the Sobolev spaces defined above, we have the following result for the  Dirichlet-Neumann operator.

\begin{lemma} [\hspace{1sp}\cite{MR3260858, MR4090462}]
Let $s>\f32$, and $\f12 \leq \mathfrak{s} \leq s$.
Consider $f \in \widetilde{H}^\mathfrak{s}_- (\R)$ and $\eta \in H^s(\R)$ with dist$(\eta, \Gamma^-) \geq h>0$.
Then $G^-(\eta)f \in H^{\mathfrak{s}-1}(\R)$ and
\begin{equation} \label{GfEtaf}
 \|G^-(\eta)f \|_{H^{\mathfrak{s} -1}} \lesssim_{\|\eta\|_{H^s}}  \|f\|_{ \widetilde{H}^\mathfrak{s}_-},
\end{equation}
where the implicit constant in the inequality depends only on $s, \mathfrak{s}, h$ and $\| \underline{b}^-\|_{\dot{W}^{1,\infty}(\R)}$.
\end{lemma}

For the two-dimensional problem, due to the simple geometry, we have the following result on the paralinearization of the Dirichlet-Neumann operator.
\begin{lemma}[\hspace{1sp}\cite{MR3260858, MR4090462}]
Let $s>\f32$, $\delta \in (0, \f12]$, and $\mathfrak{s} \in [\f12, s-\delta]$.
If $f \in \widetilde{H}^\mathfrak{s}_- (\R)$ and $\eta \in H^s(\R)$ with dist$(\eta, \Gamma^-)>h>0$, then 
\begin{equation} \label{GfDRf}
G^-(\eta)f = |D|f + R^-(\eta)f, \quad \| R^-(\eta)f\|_{H^{\mathfrak{s} - 1+\delta}} \lesssim_{\|\eta\|_{H^s}} \|f\|_{\widetilde{H}^\mathfrak{s}_-},
\end{equation}
where the implicit constant in the inequality depends only on $(s, \mathfrak{s}, \delta, h)$. 
\end{lemma}

Finally, we record the contraction estimate for the Dirichlet-Neumann operator.
This will be needed when we prove the uniqueness and the continuous dependence of the solutions.
\begin{lemma}[\hspace{1sp}\cite{MR4090462}]
Let $s>\f32$.
If $f \in \widetilde{H}^{s-\f12}_- (\R)$ and $\eta_1, \eta_2 \in H^s(\R)$ with dist$(\eta_j, \Gamma^-)>h>0$ for $j = 1,2$, then for all $\mathfrak{s} \in [\f12, s]$
\begin{equation} \label{GetaDiff}
\| G^-(\eta_1)f - G^-(\eta_2)f\|_{H^{\mathfrak{s} - 1}} \lesssim_{\|(\eta_1, \eta_2)\|_{H^s \times H^s}}\|\eta_1 - \eta_2 \|_{H^s} \|f\|_{\widetilde{H}^{\mathfrak{s}}_-},
\end{equation}
where the implicit constant in the inequality depends only on $(s, \mathfrak{s}, h)$. 
\end{lemma}

Next, we obtain the space-time estimates for fractional heat equations. 
\begin{lemma}[\hspace{1sp}\cite{MR2768550}] \label{Heates1}
There exist positive constants $c$ and $C$, such that for all $p\in [1, \infty]$, $t>0$, $\alpha_1,\nu_1>0$, $\alpha_2,\nu_2\geq0$ and $j\in\mathbb{Z}$, we have
\begin{equation*}
 \|e^{-t(\nu_1|D_x|^{\alpha_1}+\nu_2|D_x|^{\al_2})}P_ju\|_{L^p(\R)} \leq C e^{-c(\nu_1 2^{\alpha_1j}+\nu_2 2^{\al_2j})t} \|P_j u\|_{L^p(\R)}.
\end{equation*}
\end{lemma}
Note that our discussion is restricted to the case 
$p=2$, where the proof is quite simple.

By slightly modifying the proof of Proposition $2.9$ in \cite{MR4348695}, we obtain the following result for the space-time estimate for the fractional heat kernel.

\begin{proposition}
Let $s\in \R$, $\nu_1>0$, $\nu_2\geq0$, and $1\leq q_2 \leq q_1 \leq \infty$.
Let $\alpha_1>0$, $\alpha_2\geq0$, $I = [a,b]$, where $a\in \R \cup \{-\infty \}$ and $b\in \R$.
Then there exist positive constants $C_1$ and $C_2$, such that
\begin{align}
 &\|e^{-z(\nu_1  |D_x|^{\alpha_1}+\nu_2|D_x|^{\alpha_2})}u(x) \|_{\widetilde{L}^{q_1}_z(I; H^{s+\frac{\alpha_1}{q_1}})} \leq \frac{C_1}{\nu_1^{\frac{1}{q_1}}}\|u\|_{H^s},  \label{SpaceTimeOne}\\
 &\left\|\int_a^z e^{-(z-y)(\nu_1|D_x|^{\alpha_1}+\nu_2|D_x|^{\alpha_2})}f(x,y)dy \right\|_{\widetilde{L}^{q_1}_z(I; H^{s+\frac{\alpha_1}{q_1}})} \leq \frac{C_2}{\nu_1^{1+\frac{1}{q_1}-\frac{1}{q_2}}} \|f\|_{\widetilde{L}^{q_2}(I; H^{s-\alpha_1+\frac{\alpha_1}{q_2}})}. \label{SpaceTimeTwo}
\end{align}
\end{proposition}
\begin{proof}
By Lemma \ref{Heates1}, we have
\begin{align*}
\|P_je^{-z(\nu_1  |D_x|^{\alpha_1}+\nu_2|D_x|^{\alpha_2})}u\|_{L^2(\R)}\leq& C e^{-c (\nu_12^{\alpha_1 j}+\nu_22^{\al_2j} )z} \|P_j u\|_{L^2(\R)}\\
\leq& C e^{-c \nu_12^{\alpha_1 j}z} \|P_j u\|_{L^2(\R)}.
\end{align*}
We then  get \eqref{SpaceTimeOne} from
\begin{align*}
\left(\int_{a}^be^{-c \nu_12^{\alpha_1j}q_1z}dz\right)^{\f{1}{q_1}}
=(cq_1)^{-\f{1}{q_1}}\f{2^{-\f{\al_1}{q_1} j}}{\nu_1^{\f1{q_1}}}\left(\int_a^be^{-c \nu_12^{\alpha_1j}q_1z}d(c\nu_12^{\al_1}q_1z)\right)^{\f1{q_1}}\leq C_1\f{2^{-\f{\al_1}{q_1} j}}{\nu_1^{\f1{q_1}}}.
\end{align*}

By Lemma \ref{Heates1} and Young's inequality in $z$, we have
\begin{align*}
&\left\|P_j\int_a^z e^{-(z-y)(\nu_1|D_x|^{\alpha_1}+\nu_2|D_x|^{\alpha_2})}f(x,y)dy \right\|_{L^{q_1}_z(I; L^2(\R))}\\
\leq& C\left\|\int_a^b e^{-c (\nu_12^{\alpha_1 j}+\nu_22^{\al_2j} )(z-y)}\|P_jf(\cdot,y)\|_{L^2}dy \right\|_{L^{q_1}_z(I)}\\
\leq& C\left\|\int_a^b e^{-c\nu_12^{\alpha_1 j}(z-y)}\|P_jf(\cdot,y)\|_{L^2}dy \right\|_{L^{q_1}_z(I)} \leq \f{C_2}{\nu_1^{1+\f1{q_1}-\f{1}{q_2}}}2^{\al_1j\big(-1+\f1{q_2}-\f{1}{q_1}\big)}\|P_jf\|_{L^{q_2}_z(I;L^2)},
\end{align*}
which implies \eqref{SpaceTimeTwo} by summing $j$ from $0$ to $\infty$.
\end{proof}
These results are used with either $\alpha_1 = 5, \alpha_2 = 1$ or $\alpha_1 = 1, \alpha_2 = 0$.

To obtain solutions of equations as fixed points of the mappings, we record here two fixed-point results from \cite{MR4348695}.
\begin{lemma}[\hspace{1sp}\cite{MR4348695}] \label{t:fixedpoint}
Let $(E, \|\cdot\|)$ be a Banach space, and let $\nu >0$.
Denote by $B_\nu$ the closed ball of radius $\nu$ centered at $0$ in $E$. 
Assume that $\mathcal{B}:E\times E \rightarrow E$ and there exists $\mathcal{F}: \R^+ \rightarrow \R^+$ such that
\begin{equation*} %\label{NuFTwoNu}
\nu \mathcal{F}(2\nu) \leq \f12,
\end{equation*}
and the following two conditions hold:
\begin{itemize}
\item For all $x\in B_\nu$, $\mathcal{B}(x, \cdot)$ is linear and
\begin{equation*} %\label{BUniformBound}
 \|\mathcal{B}(x, y)\| \leq \mathcal{F}(\|x\|)\|x\| \|y\|, \quad \forall y\in B_\nu.
\end{equation*}
\item For all $x_1, x_2, y \in B_\nu$,
\begin{equation*} %\label{BDiffBound}
 \|\mathcal{B}(x_1, y) - \mathcal{B}(x_2, y) \| \leq \mathcal{F}(\|x_1\| + \|x_2\|)\|x_1 - x_2\| \|y\|.
\end{equation*}
Then there exists $\delta = \delta(\nu, \mathcal{F})>0$ small enough such that, for all $x_0 \in E$ with norm less than $\delta$, $x \mapsto x_0 + \mathcal{B}(x,x)$ has a unique fixed point $x_*$ in $B_\nu$ with $\|x_*\| \leq 2\|x_0\|$.
\end{itemize}
\end{lemma}

\begin{lemma}[\hspace{1sp}\cite{MR4348695}]\label{t:fp}
Let $E_1$ and $E_2$ be two normed spaces such that $E_1$ is complete. Assume that $E_2$ has the Fatou property: if $(u_n)$ is a bounded sequence in $E_2$ then there exist $u\in E_2$ and a subsequence $(u_{n_k})$ such that $u_{n_k}\to u$ in a sense weaker than norm convergence in $E_1$ and that 
\[
\| u\|_{E_2}\le C\liminf_{k\to \infty} \| u_{n_k}\|_{ E_2}.
\]
 Assume that $K:E_1\to E_1$ is a linear operator such that $K: E_1\cap E_2\to E_2$ and the following property holds.  There exist $(\alpha_1, \alpha_2)\in (0, 1)^2$ and $A>0$ such that for all $u\in E_1\cap E_2$, 
\begin{align*}
&\| K(u)\|_{E_1}\le \alpha_1 \| u\|_{E_1}, \\
&\|K(u)\|_{E_2}\le \alpha_2\| u\|_{E_2}+A \| u\|_{E_1}. 
\end{align*}
Then for any $u_0\in E_1\cap E_2$ there exists a unique fixed point $u_*\in E_1\cap E_2$ of the mapping $u\mapsto u_0+K(u)$.
\end{lemma}

Finally, following the idea in Proposition $3.4$ in \cite{MR4348695}, we prove a refined remainder estimate of the Dirichlet-Neumann operator at higher regularity.
\begin{lemma}\label{t:elliptic}
Let $\eta ,f\in  H^{1+r}(\R)$ for $r>\f12$.
Suppose the lower boundary $\Gamma^-$ is either empty or flat $(\underline{b}^{-}_x = 0)$.
There exists a constant $c_1<1$ such that if
$\|\eta\|_{W^{1+\epsilon,\infty}}<c_1$ for $\epsilon$ sufficiently small, then
\begin{equation} \label{GfDRf2}
G^-(\eta)f = |D|f + R^-(\eta)f, \quad \| R^-(\eta)f\|_{H^{r}} \lesssim \| \eta\|_{W^{1+\epsilon,\infty}} \|f\|_{H^{r+1}} + \| f\|_{W^{1,\infty}} \|\eta\|_{H^{r+1}}.
\end{equation}
\end{lemma}

\begin{proof}
 For simplicity, we assume that $\Gamma^- = \emptyset$, so that the domain $\Omega^-$ has infinite depth.
 The analysis for the flat lower boundary case is similar, and one just needs to replace $H^{1+r}$ by $\widetilde{H}^{1+r}_-$.
 We fix the time $t$, and consider the elliptic problem in $\Omega^-$:
 \begin{equation}\label{eq:elliptic}
\begin{cases}
\Delta_{x, y} \phi=0\quad\text{in}~\Omega^-,\\
\phi=f\quad\text{on}~\Sigma,\\
\nabla_{x, y}\phi\to 0\quad\text{as}~y\to -\infty.
\end{cases}
\end{equation}
Assume that $\phi$ is a smooth solution of \eqref{eq:elliptic}. For $J=(-\infty, 0)$, we straighten the free boundary  using the change of variables $\R\times J\ni (x, z)\mapsto (x, \varrho(x, z))\in \Omega^-$, where
\begin{equation*}%\label{def:rho}
\varrho(x, z)=z+H(x, z),\quad H(x, z)=e^{z|D_x|}\eta(x),\quad (x, z)\in \R\times J.
\end{equation*}
Clearly, $\varrho(x, 0)=\eta(x)$ and $\varrho(x,z)\to -\infty$ as $z\to -\infty$.
Using \eqref{SpaceTimeOne}, we get
\begin{equation} \label{HRhoBound}
 \|H \|_{\widetilde{L} ^1(J; H^{r+5})} \lesssim \| \eta\|_{ H^{r+4}}, \quad  \|H \|_{\widetilde{L} ^\infty(J; H^{r})} \lesssim \| \eta\|_{ H^{r}}.
\end{equation}

Since $\phi$ is harmonic in $\Omega^-$, a direct computation shows that $v(x, z)=\phi(x,\varrho(x, z))$ solves the elliptic equation
\begin{equation}\label{eq:Deltav}
(\p_z+|D_x|)(\p_z-|D_x|)v=\partial_zQ_a[v]+|D_x|Q_b[v]\quad\text{in} ~\R \times J,
\end{equation}
 where $Q_a[v]$ and $Q_b[v]$ are given by
\begin{equation*}%\label{Qab}
\begin{aligned}
&Q_a[v]=\partial_xH\cdot \partial_xv-\frac{|\partial_xH|^2-|D_x|H}{1+|D_x|H}\partial_zv,\\
&Q_b[v]=|D_x|^{-1}\partial_x\big(\partial_xH\partial_zv-|D_x|H\partial_xv\big).
\end{aligned}
\end{equation*}

From the equation \eqref{eq:Deltav}, we get that $v$ is a fixed point of the operator 
\begin{equation}\label{def:cT}
\begin{aligned}
\mathcal{T}[v](x, z)&=e^{z|D_x|}f(x)+\int_0^ze^{(z-z')|D_x|}Q_a[v](x, z')dz'\\
&+\int_0^ze^{(z-z')|D_x|}\int_{-\infty}^{z'}e^{-(z'-\tau)|D_x|}|D_x|\{Q_b[v](x, \tau)-Q_a[v](x, \tau)\}d\tau dz'.
\end{aligned}
\end{equation}
If $v$  is a fixed point of $\mathcal{T}$ \eqref{def:cT}, then 
 the Dirichlet-Neumann operator $G^-(\eta)f$ is defined by
\begin{equation}\label{def:R}
G^-(\eta)f=|D_x|f+R^-(\eta)f:=|D_x|f+\int_{-\infty}^0e^{\tau |D_x|}|D_x|\{Q_b[v](x, \tau)-Q_a[v](x, \tau)\}d\tau.
\end{equation}
Define the auxiliary function
\begin{equation*}
    w(x,y) : = \int_{-\infty}^ye^{-(y-\tau)|D_x|}|D_x|\{Q_b[v](x, \tau)-Q_a[v](x, \tau)\}d\tau,\quad y\le 0.
\end{equation*}
The estimate \eqref{SpaceTimeTwo} with $\nu_1 = \alpha_1 = 1$, and $\nu_2 = \alpha_2 = 0$ shows that
\begin{equation}\label{Xest:w10}
 \| w\|_{X^r(J)}\le C \| |D_x|(Q_b-Q_a)\|_{\widetilde{L}^1(J;  H^{r+4})}\lesssim \| (Q_a, Q_b)\|_{\widetilde{L} ^1(J; H^{r+5}\times H^{r+5})}.
\end{equation}
We can also write the operator $\mathcal{T}$ as 
\begin{equation*}%\label{def:opK}
\mathcal{T}[v](x, z)=e^{z|D_x|}f(x)+K[v](x, z):=e^{z|D_x|}f(x)+\int_0^ze^{(z-z')|D_x|}\{w(x, z')+Q_a(x, z')\}dz'.
\end{equation*}
We have the estimate
\begin{equation*}
 \||D_x| (K[v]) \|_{X^r(J)} \lesssim \| w\|_{\widetilde{L} ^1(J; H^{r+5})} + \|Q_a\|_{\widetilde{L} ^1(J; H^{r+5})} \lesssim \| (Q_a, Q_b)\|_{\widetilde{L} ^1(J; H^{r+5}\times H^{r+5})}.
\end{equation*}
We compute the partial derivative with respect to $z$ of $K[v]$:
\begin{equation*}
\label{form:dzK}
\p_z K[v](x, z)=\int_0^ze^{(z-z')|D_x|}|D_x|\{w(x, z')+Q_a(x, z')\}dz'+w(x, z)+Q_a(x, z),
\end{equation*}
so that by using estimates \eqref{SpaceTimeTwo} and \eqref{Xest:w10}, 
\begin{align*}
 \|\p_z K[v](x, z)\|_{\wt L^1(J; H^{r+5})}&\lesssim\| w\|_{\wt L^1(H^{r+5})}+\| Q_a\|_{\wt L^1(J; H^{r+5})}\lesssim \| (Q_a, Q_b)\|_{\wt L^1(J; H^{r+5}\times H^{r+5})}, \\
  \|\p_z K[v](x, z)\|_{\wt L^\infty(J; H^{r})}&\lesssim \| w\|_{\wt L^1(H^{r+1})}+\| Q_a\|_{\wt L^1(J; H^{r+1})} + \| w\|_{\wt L^\infty (J;H^{r})} + \|Q_a\|_{\wt L^\infty(J; H^{r})} \\
 &\lesssim \| (Q_a, Q_b)\|_{\wt L^1(J; H^{r+5}\times H^{r+5})} + \|Q_a\|_{ \wt L^\infty(J; H^{r})}.
\end{align*}
Hence, we get that
\begin{align*}
\|( |D_x| K[v],\p_z K[v])\|_{X^r(J) \times X^r(J)} &\lesssim \| (Q_a, Q_b)\|_{\wt L^1(J; H^{r+5}\times H^{r+5})} + \|Q_a\|_{ \wt L^\infty(J; H^{r})} \\
&\lesssim  \|Q_a\|_{X^r(J)} +  \| Q_b\|_{\wt L^1(J;  H^{r+5})}.
\end{align*}

To estimate $Q_a[v]$, we need to consider the term
\begin{align*}
    I =&  \frac{|\partial_xH|^2-|D_x|H}{1+|D_x|H}\partial_zv = I_1 + I_2+ I_3 \\
    =& |\partial_xH|^2 \partial_zv - |\partial_xH|^2F(|D_x|H) \partial_zv  - F(|D_x|H)\partial_zv, \quad F(x) = \frac{x}{1+x}.
\end{align*}
We estimate $I$ using \eqref{HsLinftyEst}, Moser estimate \eqref{MoserOne}, and the fact that $|F|\leq 1$
\begin{align*}
 \|&I_1\|_{X^r(J)} \lesssim \|\p_x H \|^2_{L^\infty(J; L^\infty)} \|\p_z v\|_{X^r(J)} + \| \p_x H \|_{X^r(J)} \| \p_x H \|_{L^\infty(J; L^\infty)}\|\p_z v\|_{L^\infty(J; L^\infty)} \\
 &\lesssim \|\eta_x\|^2_{L^\infty} \|\p_z v\|_{X^r(J)} + \|\eta_x\|_{L^\infty} \|\eta\|_{H^{r+4}} \|\p_z v\|_{L^\infty(J; L^\infty)}, \\
 \|&I_2\|_{X^r(J)} \lesssim  \|\p_x H \|^2_{L^\infty(J; L^\infty)} \|\p_z v\|_{X^r(J)} + \| \p_x H \|_{X^r(J)} \| \p_x H \|_{L^\infty(J; L^\infty)}\|\p_z v\|_{L^\infty(J; L^\infty)} \\
 &+ \|\p_x H \|^2_{L^\infty(J; L^\infty)} \|\p_z v\|_{L^\infty(J; L^\infty)}  \|F(|D_x|H) \|_{X^r(J)}\\
 &\lesssim \|\eta_x\|^2_{L^\infty} \|\p_z v\|_{X^r(J)} + \|\eta_x\|_{L^\infty} \|\eta\|_{H^{r+4}} \|\p_z v\|_{L^\infty(J; L^\infty)} + \|\eta_x\|^2_{L^\infty} \|\p_z v\|_{L^\infty(J; L^\infty)} \|\eta\|_{H^{r+5}}, \\
 \|&I_3\|_{X^r(J)} \lesssim \|F(|D_x|H) \|_{L^\infty(J; L^\infty)} \|\p_z v\|_{X^r(J)} + \|F(|D_x|H) \|_{X^r(J)} \|\p_z v\|_{L^\infty(J; L^\infty)} \\
 &\lesssim \|\eta\|_{W^{1+\epsilon,\infty}}  \|\p_z v\|_{X^r(J)} + \| \eta\|_{H^{r+5}}   \|\p_z v\|_{L^\infty(J; L^\infty)}. 
\end{align*}
Other terms in $Q_a$ and $Q_b$ are estimated similarly, so that since $\|\eta\|_{W^{1,\infty}}< 1$,  
\begin{align*}
 \|( |D_x| K[v],\p_z K[v])\|_{X^r(J) \times X^r(J)} \lesssim&\|Q_a\|_{X^r(J)} +  \| Q_b\|_{\wt L^1(J;  H^{r+5})} \\
 \lesssim& \|\eta\|_{W^{1+\epsilon,\infty}} \|\nabla_{x,z} v\|_{X^r(J)} + \| \eta\|_{H^{r+5}}   \|\nabla_{x,z}  v\|_{L^\infty(J; L^\infty)}.
\end{align*}
We define the space
\begin{align*}
 X_*(J)& = L^\infty(J; L^\infty) \cap \wt L^1(J;  W^{1,\infty}), \quad  \mathcal{V}_*=\{v\in \mathcal{D}'(\R\times J): (|D_x|v, \p_zv)\in X_*(J)\}/\R, \\
  \mathcal{V}^r&=\{v\in \mathcal{D}'(\R\times J): (|D_x|v, \p_zv)\in   X^r(J)\}/\R.
\end{align*}
Then one has the estimate 
\begin{equation*}
 \|K[v] \|_{ \mathcal{V}^r} \lesssim \|\eta\|_{W^{1+\epsilon,\infty}} \| v\|_{\mathcal{V}^r(J)} + \| \eta\|_{H^{r+5}}   \|  v\|_{\mathcal{V}_*(J)},
\end{equation*}
and similarly, one can get the estimate
\begin{equation*}
 \|K[v] \|_{ \mathcal{V}_*} \lesssim \|\eta\|_{W^{1+\epsilon,\infty}} \| v\|_{\mathcal{V}_*(J)}.
\end{equation*}

Hence, according to Lemma \ref{t:fp}, using $E_1 = \mathcal{V}_*$ and $E_2 = \mathcal{V}^r$, by choosing $c_1$ small enough, $\mathcal{T}$ has a unique fixed point $v \in \mathcal{V}_* \cap  \mathcal{V}^r$.
Since $v- K[v] = e^{z|D_x|}f$, we have
\begin{align*}
   &\|\nabla_{x,z} v\|_{\wt L^\infty (J; L^\infty)}\lesssim \|\nabla_{x,z}e^{z|D_x|}f\|_{\wt L^\infty(J;L^\infty)} \lesssim \|f\|_{W^{1,\infty}}, \\
   &\|\nabla_{x,z} v\|_{\wt L^\infty (J; H^r)}\lesssim \|\nabla_{x,z}e^{z|D_x|}f\|_{\wt L^\infty(J;H^r)} \lesssim \|f\|_{H^{r+1}}.
\end{align*}
Note that using the definition of $R^-(\eta) f$ in \eqref{def:R}, and the estimate \eqref{SpaceTimeTwo},
\begin{align*}
 &\|R^-(\eta)f\|_{H^r} \lesssim \|(Q_a[v], Q_b[v]) \|_{\wt L^\infty (J; H^{r})}\\
 \lesssim &\|\eta\|_{W^{1+\epsilon,\infty}} \|\nabla_{x,z} v\|_{\wt L^\infty (J; H^{r})} + \|\eta\|_{H^{r+1}}  \|\nabla_{x,z} v\|_{\wt L^\infty (J; L^\infty)}\\
 \lesssim &\|\eta\|_{W^{1+\epsilon,\infty}} \|f\|_{ H^{r+1}} + \|\eta\|_{H^{r+1}}  \|f\|_{W^{1,\infty}}.
\end{align*}
This gives the estimate for the remainder term of the Dirichlet-Neumann operator.
\end{proof}

\textbf{Acknowledgments.}
Lizhe Wan is partially supported by the National Natural Science Foundation of China under Grant No.~12671283.
Jiaqi Yang is supported by the National Natural Science Foundation of China under Grant No.~12471225 and by the Natural Science Basic Research Program of Shaanxi under Program No.~2026JC-YXQN-021.

\bibliography{HWW}
\bibliographystyle{plain}
	
\end{document}